\documentclass[12pt,letter,twoside]{article}

\usepackage{lineno}
\usepackage{geometry}              
\usepackage[affil-it]{authblk}
\usepackage{graphicx}
\usepackage{amsmath, amssymb,amsthm}
\usepackage{enumerate}
\usepackage{comment}
\usepackage{multirow}
\usepackage{url}
\usepackage{color}
\usepackage[pdftex,colorlinks=black,citecolor=black,linkcolor=black]{hyperref}
\usepackage{epstopdf}
\usepackage{listings} 
\usepackage{algorithmic,algorithm}
\usepackage{setspace}

\newbox\Ea
\setbox\Ea=\hbox{\raise0.9pt\hbox{$=$}}

\def\ec{\mathrel{\hbox{$\copy\Ea\kern-\wd\Ea\raise-3.5pt\hbox{$\sim$}$}}}

\newcommand{\bx}{\mathbf{x}}

\newcommand{\bu}{\mathbf{u}}
\newcommand{\bv}{\mathbf{v}}
\newcommand{\bz}{\mathbf{z}}
\newcommand{\by}{\mathbf{y}}

\newcommand{\bI}{\mathbf{I}}
\newcommand{\bw}{\mathbf{w}}
\newcommand{\bn}{\mathbf{n}}
\newcommand{\bP}{\mathbf{P}}

\newcommand{\bsigma}{\boldsymbol{\sigma}}
\newcommand{\bepsilon}{\boldsymbol{\epsilon}}

\newcommand{\bphi}{\boldsymbol{\phi}}

\newcommand{\cA}{{\mathcal A}}
\newcommand{\cB}{{\mathcal B}}
\newcommand{\cK}{{\mathcal K}}
\def\ec{\mathrel{\hbox{$\copy\Ea\kern-\wd\Ea\raise-3.5pt\hbox{$\sim$}$}}}

\newtheorem*{remark}{Remark}
\newtheorem{lemma}{Lemma}

\newtheorem{theorem}{Theorem}

\newtheorem{coro}{Corollary}

\begin{document}

\title{Well-posedness and Robust Preconditioners \\
for the Discretized Fluid-Structure Interaction Systems\thanks{
 This work was supported in part by NSF Grant DMS-1217142, DOE Grant DE-SC0006903 and Yunan Provincial Science and Technology Department Research  Award: Interdisciplinary Research in Computational Mathematics  and Mechanics with Applications in Energy Engineering. 
}}
\author{Jinchao Xu%
  \thanks{Email address: \texttt{xu@math.psu.edu}; 
 }}
   \affil{The Center for Computational Mathematics and Applications \\and Department of Mathematics, \\Pennsylvania State University,\\ University park, PA 16802}
\author{Kai Yang
 \thanks{Email address: \texttt{yang\_k@math.psu.edu}; Corresponding author. }}
   \affil{The Center for Computational Mathematics and Applications \\and Department of Mathematics, \\Pennsylvania State University,\\ University park, PA 16802}
%\institute{Department of Mathematics, Pennsylvania State University, University park, PA 16802, \texttt{xu@math.psu.edu}\and Department of Mathematics, Pennsylvania State University, University park, PA 16802, \texttt{yang\_k@math.psu.edu} }
\maketitle


\begin{abstract}

In this paper we develop a family of preconditioners for the
linear  algebraic systems arising from the arbitrary Lagrangian-Eulerian discretization of some
  fluid-structure interaction models.  After the time discretization, we formulate the fluid-structure interaction equations as saddle point problems and prove the uniform well-posedness. Then we discretize the space dimension by finite element methods and prove their uniform well-posedness by two different approaches under appropriate assumptions. The uniform well-posedness  makes it possible to design robust preconditioners for the  discretized fluid-structure interaction systems. Numerical examples are presented to show the robustness and efficiency of these preconditioners.
  
%  
%  
%  Our major observation is that a
%  straightforward finite element discretization based on an arbitrary Lagrangian-Eulerian
%  formulation, as often used in most existing literature, leads to a
%  discretized system that is not uniformly stable with respect to material and
%  discretization parameters.  To overcome this hidden instability, we propose two approaches. The first one is to introduce a stabilization term to the fluid
%  equations and the second one is to adopt a norm of velocity space that depends on the choice of pressure space.  Both of these approaches lead to uniform well-posedness
%  of finite element discretization of FSI model under appropriate assumptions.  This
%  uniform well-posedness makes it possible to design robust
%  preconditioners for the coupled FSI discrete system.  Numerical
%  examples are presented to show the robustness and efficiency of these
%  preconditioners.

\end{abstract}

%\linenumbers

{

\section{Introduction}
Fluid-structure interaction (FSI) is a much studied topic aimed at understanding the interaction between some moving structure and fluid and how their interaction affects the interface between them. FSI has a wide range of applications in many areas including hemodynamics \cite{Formaggia.L;Quarteroni.A;Veneziani.A2009b,Quarteroni.A2010a,Quarteroni.A;Veneziani.A;Zunino.P2001a,Crosetto.P;Deparis.S;Fourestey.G;Quarteroni.A2010a} and wind/hydro turbines \cite{Bazilevs.Y;Takizawa.K;Tezduyar.Ta,Hsu.M;Bazilevs.Y2012a,Bazilevs.Y;Hsu.M;Kiendl.J;Wuchner.R;Bletzinger.K2011a,Bazilevs.Y;Hsu.M;Akkerman.I;Wright.S;Takizawa.K;Henicke.B;Spielman.T;Tezduyar.T2011a}.

%The essential difficulty lies in the different approaches of modeling and simulation for solid and fluid.
FSI problems are computationally challenging. The computational domain of FSI consists of fluid and structure subdomains. The position of the interface between fluid domain and structure domain is time dependent. Therefore, the shape of the fluid domain is one of the unknowns, increasing the nonlinearity of the FSI problems.

Many numerical approaches have been proposed to tackle the interface problem of FSI. The arbitrary Lagrangian-Eulerian (ALE) method is commonly used.  ALE adapts the fluid mesh to match the displacement of structure on interface.  Other approaches, such as the fictitious domain method \cite{Glowinski.R;Pan.T;Hesla.T;Joseph.D;Periaux.J2001a,Yu.Z2005a} and the immersed boundary method \cite{Zhang.L;Gerstenberger.A;Wang.X;Liu.W2004a,Wang.H;Chessa.J;Liu.W;Belytschko.T2008a,Peskin.C2002a},  have inconsistent fluid and structure meshes and, therefore, need special treatment at the interface, such as interpolation between different meshes. In this paper, we focus on the ALE method.

There is much research focused solving fluid-structure interaction problem numerically using ALE formulation. These studies can be roughly classified into partitioned approaches and monolithic approaches \cite{Dunne.T;Rannacher.R;Richter.T2010a}. Partitioned approaches employ single-physics solvers to solve the fluid and structure problems separately and then couple them by the interface conditions.  Monolithic approaches solve the fluid and structure problems simultaneously.  Depending on whether the interface conditions are exactly enforced at every time step, these approaches can also be classified into weakly and strongly coupled algorithms. Weakly coupled partitioned approaches are usually considered unstable due to the added-mass effect  \cite{Causin.P;Gerbeau.J;Nobile.F2005a}. A semi-implicit approach proposed in \cite{Fernandez.M;Gerbeau.J;Grandmont.C2007a} can avoid the added-mass effect for a wide range of applications, but it is subject to pressure boundary conditions.  Several types of semi-implicit methods were proposed in \cite{Quaini.A;Quarteroni.A2007a,Murea.C;Sy.S2009a}.  Strongly coupled approaches are  preferred for their stability. Although it is possible to achieve the strong coupling via partitioned solvers (by fixed-point iteration, for example), they usually introduce prohibitive computational costs due to slow convergence \cite{Fernandez.M;Gerbeau.J2009a}. In this paper we consider strongly coupled monolithic approaches and address some solver issues.  Monolithic approaches give us larger linear systems, for which efficient solvers are needed. 

 A great deal of work has been carried out to develop monolithic solvers for FSI \cite{Gee.M;Kuttler.U;Wall.W2011a,Turek.S;Hron.J2010a,Cai.X2010a,Barker.A;Cai.X2009a}.  In \cite{Heil.M2004a},  a fully-coupled solution strategy is proposed to solve the FSI problem with large structure displacement. The nonlinearity is handled by Newton's method and various approaches to solve the Jacobian system are proposed. Block triangular preconditioners and pressure Schur complement preconditioners are used for the preconditioned Krylov subspace solvers. However,  in \cite{Gee.M;Kuttler.U;Wall.W2011a} it is pointed out that block preconditioning for fluid and structure separately cannot resolve the coupling between fields and it is proposed that structure degrees of freedoms on interface  be eliminated in order to effectively precondition degrees of freedom at the interface.   In \cite{Barker.A;Cai.X2009a,Barker.A;Cai.X2010a,Barker.A;Cai.X2010b}, a Newton-Krylov-Schwarz method for FSI is developed.  Additive Schwarz preconditioners are used for Krylov subspace solvers and two-level methods are also developed.  In \cite{Badia.S;Quaini.a;Quarteroni.a2008b,Badia.S;Quaini.A;Quarteroni.A2008a}, ILU preconditioners and inexact block-LU preconditioners are proposed to solve FSI problems.

 In this paper, we reformulate semi-discretized systems of FSI as saddle point problems with fluid velocity, pressure and structure velocity as unknowns.  The ALE mapping is decoupled from the solution of the velocity and pressure. Then, we carry out our theoretical analysis and solver design under this framework. With particular choice of norms, we prove that the saddle point problem is well-posed.

%A straightforward finite element discretization of FSI based on an ALE formulation, as often used in most existing literature, leads to a discretized system that is not uniformly stable with respect to material and mesh parameters. This is due to the fact that for finite element pairs of velocity and pressure the discrete divergence free space is not necessarily a subspace of the continuous divergence free space.   To correct this instability, 

For the finite element discretization of FSI, we propose two approaches to prove the well-posedness. The first  introduces a stabilization term to the fluid
  equations and the second adopts a norm of the velocity space that depends on the choice of the pressure space.  Both of these approaches lead to uniform well-posedness
  of the finite element discretization of the FSI model under appropriate assumptions.

Based on the uniform well-posedness, we propose optimal preconditioners based on the framework in  \cite{Mardal.K;Winther.R2011a,Zulehner.W2011a} such that the preconditioned linear systems have uniformly bounded condition numbers.  Then, we compare the proposed preconditioners with the augmented Lagrangian preconditioners \cite{Benzi.M;Olshanskii.M;Wang.Z2011a,Benzi.M;Olshanskii.M2006a,Benzi.M;Olshanskii.M2011a,Olshanskii.M;Benzi.M2008a}. To test the preconditioners, we solve the linear systems coming from the discretization of the Turek and Hron benchmark problems \cite{Turek.S;Hron.J2006a}. The iteration counts of GMRes with several preconditioners are compared.

The rest of this paper is organized as follows. In section 2, we introduce an FSI model and the ALE method.  In section 3, we study the proposed time and space discretization and its well-posedness.  In section 4, we propose optimal preconditioners for the discretized systems and demonstrate their performance with numerical examples.

%discuss on several solution techniques and application of multigrid methods.

%\input{fsi_model}

\section{An FSI model}%Kai: no ALE is involved in this section.
We consider a domain $\Omega\subset \mathbb{R}^N(N=2,3)$ with a fluid
occupying the upper half $\Omega_f$ and a solid occupying the lower half
$\Omega_s$, as illustrated in Figure \ref{fig:moving_fsi_domain}.

% figure to add
\begin{figure}
\begin{center}
\graphicspath{{:figure:}}
\includegraphics[width=0.6\textwidth]{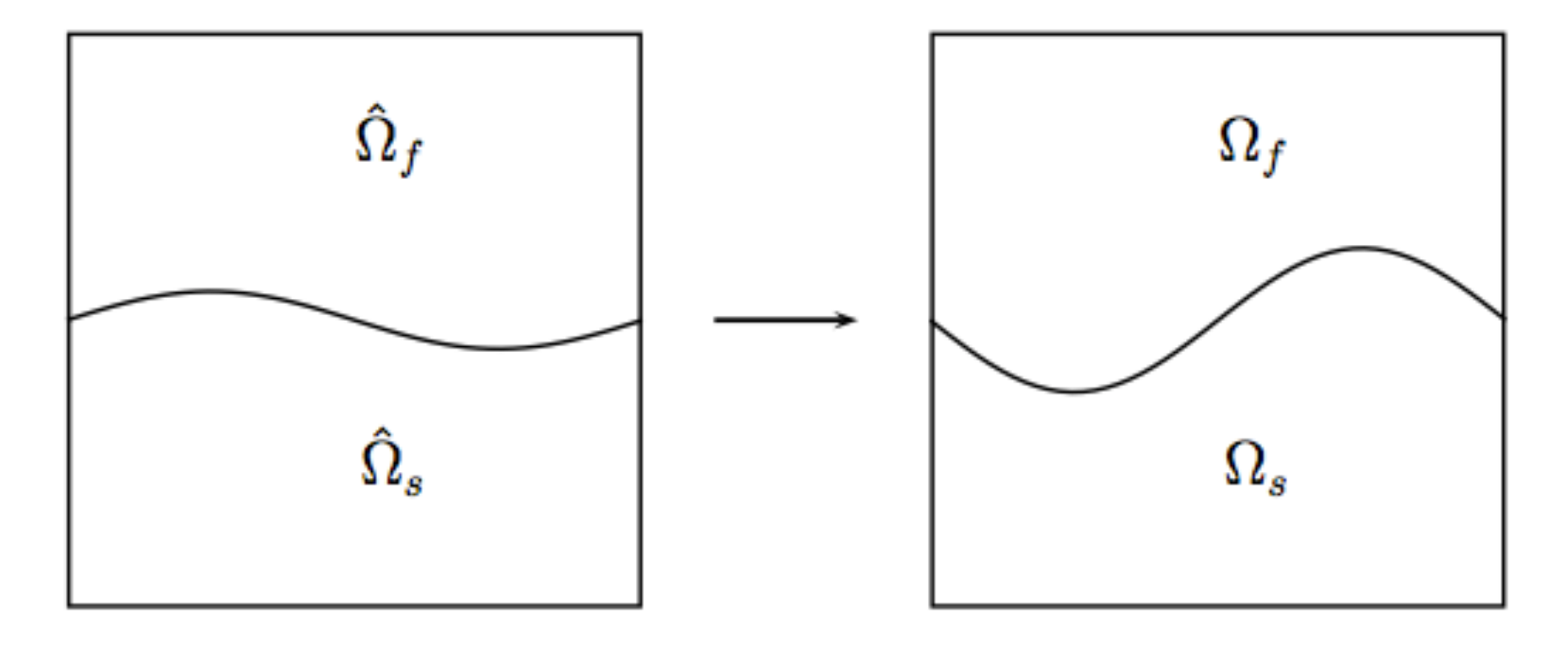}  
\end{center}
\caption{Moving domains of FSI}
\label{fig:moving_fsi_domain}
\end{figure}

%%  KY: the pstricks code for the picture
%
%\begin{figure}[h]
%\begin{center} 
% \begin{pspicture}*[0.5](0,0)(20,10)
%%\pspolygon[\linewidth=1pt](8,0)(8,4)(13,4)(13,0)
%\psframe(3,3)(7,7)
%\psplot[algebraic]{3}{7}{0.2*sin(3.14159*(x-3)/2)+5}
%%\psaxes[Dx=1,Dy=1]{->}(0,0)(-5.2,-6.2)(5,5)
%\rput[br]{0}(5.5,6){$\hat{\Omega}_s$}
%\rput[br]{0}(5.5,3.8){$\hat{\Omega}_f$}
%
%\psline[linewidth=1pt,linearc=0]{->}(7.5,5)(8.5,5)
%
%\psframe(9,3)(13,7)
%\psplot[algebraic]{9}{13}{0.5*sin(-3.14159*(x-9)/2)+5}
%\rput[br]{0}(11.5,6){${\Omega}_s$}
%\rput[br]{0}(11.5,3.8){${\Omega}_f$}
%
%\end{pspicture}
%\end{center}
%\caption{Moving domains of FSI}
%\label{fig:moving_fsi_domain}
%\end{figure}

Let $\Gamma:=\partial\Omega_f\cap\partial \Omega_s$ be the interface
of the fluid domain and the solid domain. On the outer boundary of the solid
$\partial\Omega_s\backslash \Gamma$, the solid is clamped; namely, the
displacement of the solid is zero on $\partial\Omega_s\backslash
\Gamma$. In this
paper, we always assume that both $\partial\Omega_s\backslash \Gamma$ and
$\partial\Omega_f\backslash \Gamma$ have positive measures.

%We define $\partial\Omega_s^D:=\partial\Omega_s\backslash
%\Gamma$.  We assume
%$\partial\Omega_f\backslash\Gamma=\partial\Omega_f^N\cup\partial\Omega_f^D$. On
%$\partial\Omega_f^D$, we have no-slip or inflow boundary conditions;
%on $\partial\Omega_f^N$, we have no-flux boundary conditions. 

 In addition, we assume that the interaction of the fluid and solid only
 occurs at the interface, and the interface $\Gamma$ may move over
 time due to this interaction.  We assume that the outer boundary is fixed. In the dynamic setting, we use $\Omega_f(t)$ and
 $\Omega_s(t)$ to denote the domains at time $t\in [0,T]$.  The
 domains satisfy $\bar\Omega=\bar\Omega_f(t)\cup \bar\Omega_s(t)$ and
 $\Gamma(t)=\partial\Omega_f(t)\cap\partial \Omega_s(t)$.
 
We denote the reference domains by
$$
\hat \Omega_f=\Omega_f(0),\quad \hat \Omega_s=\Omega_s(0)
$$
and the domains at time $t$ by
$$\Omega_f=\Omega_f(t),\quad \Omega_s=\Omega_s(t).$$

The motion in the fluid and structure can be characterized by
a \emph{flow map} $\bx(\hat\bx,t)$; namely, the position of the particle $\hat\bx$ at time $t$ is $\bx(\hat\bx,t)$.  Then, given $t>0$, $\bx(\cdot,t)$ is a diffeomorphism from $\Omega(0)$ to $\Omega(t)$.
%$$
%\bx({\hat{\bx}},t): \Omega(0)\times [0,T]\mapsto
%\Omega(t),
%$$
%which is assumed to be a differemorphism.
% 

For $({\hat{\bx}},t) \in \Omega(0)\times[0,T]$, we introduce the following
variables in \emph{Lagrangian coordinates} :
the displacement $\hat
{\bu}({{\hat{\bx}}},t)={\bx}({\hat{\bx}},t)-{\hat{\bx}}$, the velocity $\displaystyle  \hat {\bv}({\hat{\bx}},t)=\frac{\partial\hat
  {\bx}}{\partial t}$, the deformation tensor $\displaystyle {F}({\hat{\bx}},t)=\frac{\partial
  {\bx}}{\partial {\hat{\bx}}}({\hat{\bx}},t)$, 
and its determinant $\displaystyle {J}({\hat{\bx}},t)=det({F}({\hat{\bx}},t))$.  Using the
relationship $\bx=\bx(\hat \bx,t)$, we also introduce the velocity in \emph{Eulerian coordinates}: 
${\bv}({\bx},t)=\hat {\bv}({\hat{\bx}},t)$.  The symmetric part of the gradient is denoted by
$\displaystyle \bepsilon(\bv)=\frac{\nabla \bv+(\nabla \bv)^T}{2}. $

Let us now introduce a simple FSI model which consists of
the incompressible Navier-Stokes equations for the fluid (in Eulerian
coordinates) and linear elasticity equations for the structure (in
Lagrangian coordinates).

For clarity, we start with the momentum equations for fluid and solid
both in Eulerian coordinates:
$$
\rho_f D_t\bv_f-\nabla\cdot\bsigma_f=g_f,\quad\mbox{in
}\Omega_f,
$$
and
$$
\rho_s D_t\bv_s-\nabla\cdot\bsigma_s=g_s,\quad\mbox{in
}\Omega_s.
$$
Here $\bsigma_f$ and $\bsigma_s$ are the Cauchy stress tensors for fluid and
structure, respectively.  Here $D_t\bv_f$ and $D_t\bv_s$ are the material derivatives.

On the interface $\Gamma=\partial\Omega_f\cap\partial\Omega_s$,  the interface conditions are given in Eulerian coordinates as
\begin{equation}
\label{eq:interface_condition_E}
\bv_f=\bv_s\quad\mbox{and}\quad \bsigma_f\bn=\bsigma_s\bn\quad\mbox{ on }\Gamma.
\end{equation}
Note that we neglect some effects such as the surface tension in this model and
thus the stress is continuous on interface.

While we keep the Eulerian description for the fluid model,
we use the Lagrangian description for the structure.  Accordingly, we introduce the following
Sobolev spaces:
\begin{equation}
  \label{spaceV}
\mathbb{V}:=\{(\bv_f,\hat\bv_s)\in H^1_D(\Omega_f(t))\times H^1_D(\hat\Omega_s) \text{ such that }
\bv_f\circ\bx_s=\hat\bv_s, \text{ on }\hat\Gamma\},
\end{equation}
where%\footnote{Note the suggested change of notation below}
$$
H^1_D(\Omega_f(t)):=\{\bu\in (H^1(\Omega_f(t)))^N| \bu=0, \text{ on } \partial
\Omega\cap\partial\Omega_f\},
$$ 
$$
H^1_D(\hat\Omega_s):=\{\bu\in (H^1(\hat\Omega_s))^N| \bu=0, \text{ on } \partial
\Omega\cap\partial\hat\Omega_s\},
$$
and
$$
\mathbb{Q}:=L^2(\Omega_f(t)).
$$
$\mathbb{V}$ is defined for the fluid velocity in Eulerian coordinates and the structure velocity in Lagrangian coordinates. The condition $\bv_f\circ\bx_s=\hat\bv_s$ is used to enforce continuity of velocity in (\ref{eq:interface_condition_E}).  We will discuss the choice of norms for these spaces in the next section.

In order to formulate the problem weakly, we use test functions defined on $\Omega$, With the test function $\bphi\in H_0^1(\Omega)$, we first write the weak form for the fluid and structure, respectively:
$$
\int_{\Omega_f}\rho_fD_t
\bv_f\bphi d\bx+\int_{\Omega_f}\bsigma_f:\bepsilon(\bphi)d\bx-\int_{\Gamma}\bsigma_f\bn_f\cdot\bphi d\bx
=\int_{\Omega_f}g_f\bphi d\bx,
$$
$$\int_{\Omega_s}\rho_s D_t
\bv_s\bphi d\bx+\int_{\Omega_s}\bsigma_s:\bepsilon(\bphi) d\bx-\int_{\Gamma}\bsigma_s\bn_s\cdot\bphi d\bx
=\int_{\Omega_s}g_s\bphi d\bx.
$$

We add these two equations based on interface conditions (\ref{eq:interface_condition_E}):
$$\begin{aligned}
\int_{\Omega_f}\rho_f D_t
\bv_f\bphi d\bx+\int_{\Omega_f}\bsigma_f:\bepsilon(\bphi)d\bx+\int_{\Omega_s}\rho_sD_t
\bv_s\bphi d\bx&+\int_{\Omega_s}\bsigma_s:\bepsilon(\bphi) d\bx\\
=&\int_{\Omega_f}g_f\bphi d\bx+\int_{\Omega_s}g_s\bphi d\bx.
\end{aligned}
$$

By a change of coordinates $\bx=\bx(\hat{\bx},t),$ the stress term of
structure part can be written in Lagrangian coordinates
$$
\int_{\Omega_s}\bsigma_s:\bepsilon(\bphi)d\bx=\int_{\hat{\Omega}_s}
\hat\bsigma_s:\nabla_{\hat{\bx}}\hat{\bphi}F^{-1}\hat
Jd\hat{\bx}=\int_{\hat{\Omega}_s}(J\hat\bsigma_sF^{-T}):\nabla_{\hat{\bx}}\hat\bphi
d\hat{\bx},
$$ 
where $\hat\bphi(\hat\bx,t)=\bphi(\bx(\hat\bx,t),t)$ and $\hat\bsigma_s(\hat\bx,t)=\bsigma_s(\bx(\hat\bx,t),t)$.   We also change the coordinates for the inertial term and the body force term. Then, we get the following weak form of FSI
\begin{equation}
  \label{FSI-vari}
  \begin{aligned}
\int_{\Omega_f}\rho_f D_t
\bv_f\bphi +\bsigma_f:\bepsilon(\bphi) d\bx+\int_{\hat
  \Omega_s}\hat\rho_s\partial_{tt} \hat\bu_s\hat\bphi+&\bP_s:\nabla \hat\bphi d\hat\bx\\
  &=\int_{\Omega_f}g_f\bphi d\bx+\int_{\hat \Omega_s}J \hat g_s\bphi d\hat\bx,
  \end{aligned}
\end{equation}
which holds for any $\bphi\in \mathbb{V}$.  Here, the density of the structure $\hat\rho_s$ is defined as
$$\hat\rho_s(\hat\bx,t)=J(\hat\bx,t) \rho_s(\bx(\hat\bx,t),t)
$$
and $\bP_s=J\hat\bsigma_sF^{-T}$ is the first Piola-Kirchhoff stress.
By the conservation of mass, $\hat\rho_s$ is independent of $t$.

The variational formulation \eqref{FSI-vari} holds for general fluid and
structure models described by the Cauchy stresses $\sigma_f$ and
$\sigma_s$, respectively.   We now make some specific choices for
$\sigma_f$ and $\sigma_s$. 

For the fluid, we use the incompressible Newtonian model, which is
given by
\begin{equation}
\label{eq:constitutive_f}
\boldsymbol\bsigma_f= 2\mu_f\bepsilon(\bv_f)-p\mathbf{I}
\end{equation}
and
$$
\nabla\cdot \bv_f=0.
$$
For the structure, we use the linear elasticity model (for small deformations) in
Lagrangian coordinates,  which corresponds to the following
approximation:
\begin{equation}
\label{eq:constitutive_s}
\bP_s\approx \tilde\bP_s:=\mu_s\bepsilon(\hat\bu_s)+\lambda_s\nabla\cdot\hat\bu_s\bI.
\end{equation}
%Since both $\bsigma_s$ and $\bP_s$, the approximate identity ``$\approx$'' above can not be replaced by an identity ``$=$'' regardless of the choice of Cauchy stress $\bsigma_s$.  
 
\paragraph{Initial and boundary conditions}
We consider the following Dirichlet boundary conditions
\begin{equation*}
\begin{aligned}
\bv_f&=\bv_f^D, &\text{ on } &\partial \Omega_f\cap\partial\Omega,\\
%\bsigma_f \mathbf{n}&=0, &\text{ on }& \partial \Omega_f^N,\\
\hat\bu_s&=0, &\text{ on }& \partial \Omega_s\cap\partial\Omega,\\
\end{aligned}
\end{equation*}
and initial conditions

\begin{equation*}
\begin{aligned}
\hat\bu_s(0)=\bu_{s,0}, \quad \partial_t \hat\bu_s(0)&=\bu_{s,1},\quad\bv_f(0)&=\bv_{f,0}.\\
\end{aligned}
\end{equation*}

% figure to add
\begin{figure}
\begin{center}
\includegraphics[width=0.3\textwidth]{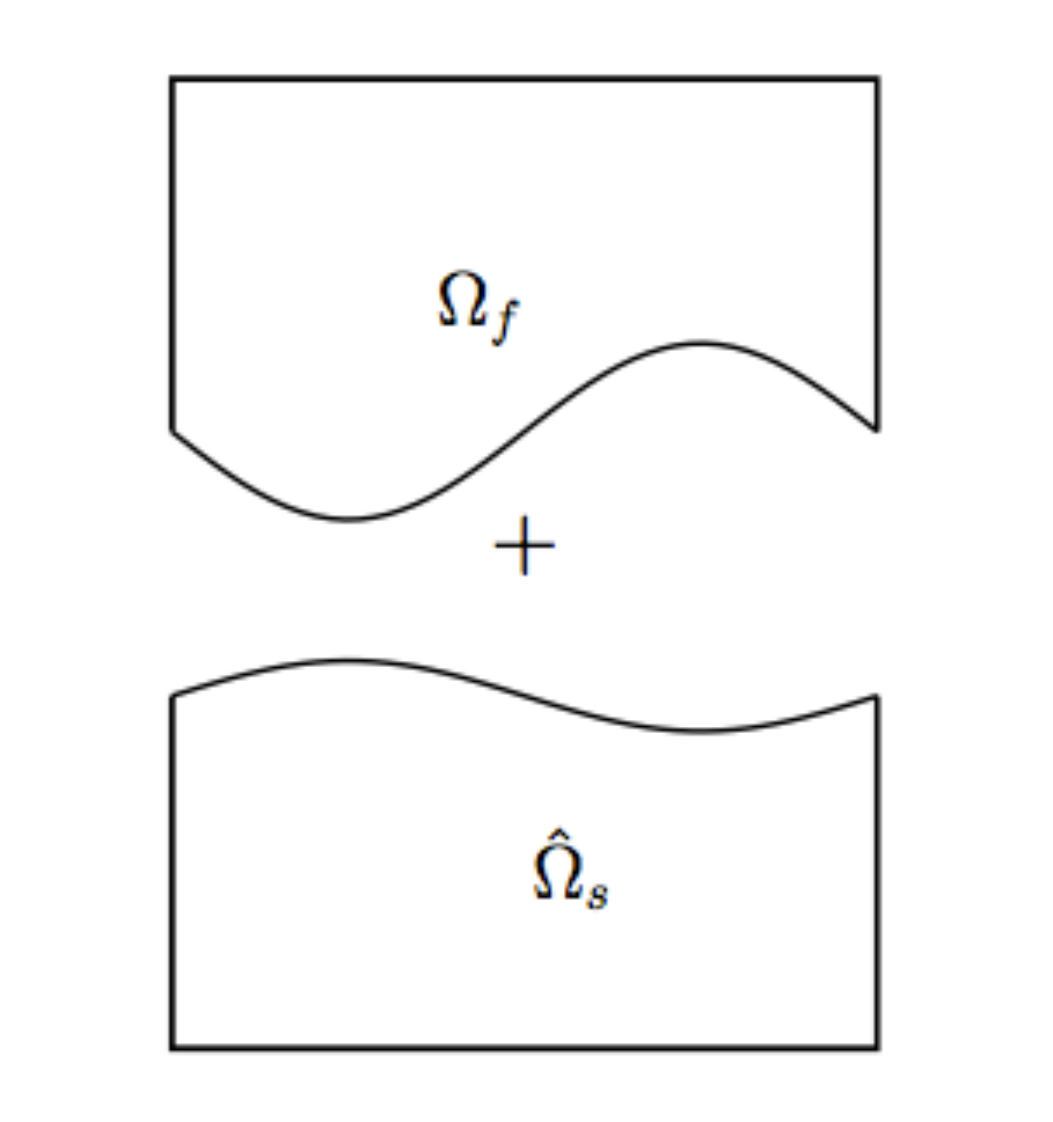}  
\end{center}
\caption{Computational domains of FSI}
\label{fig:computational_domain}
\end{figure}

%  KY: the pstricks code for the picture
%\begin{figure}[h]
%\begin{center} 
% \begin{pspicture}*[2](2,-1)(10,20)
%\psplot[algebraic]{3}{7}{0.2*sin(3.14159*(x-3)/2)+6.5}
%\psline[linewidth=1pt,linearc=0]{}(3,6.5)(3,8.5)(7,8.5)(7,6.5)
%\rput[br]{0}(5,7){$\hat{\Omega}_s$}
%
%\rput[br]{0}(5.2,5.7){\large$+$}
%
%\psline[linewidth=1pt,linearc=0]{}(3,5)(3,3)(7,3)(7,5)
%\psplot[algebraic]{3}{7}{0.5*sin(-3.14159*(x-3)/2)+5}
%\rput[br]{0}(5.5,3.8){${\Omega}_f$}
%\end{pspicture}
%\end{center}
%\caption{Computational domains of FSI}
%\label{fig:computational_domain}
%\end{figure}
In the rest of this paper, we do not rewrite the initial conditions in the weak formulations for brevity.  Moreover, we assume $\bv_f^D=0$. That is, there are only homogeneous Dirichlet boundary conditions for the fluid problem.

Together with the continuity equation and interface condition, the weak formulation of FSI is as follows:

\begin{description}
\item[The weak formulation of FSI:]
Find $\bv_f$, $p$ and $\hat\bu_s$ such that for any given $t>0$, the following equations hold for any  $(\bphi,\hat\bphi)\in \mathbb{V}$ and $ q\in \mathbb{Q}$
\begin{equation}
\label{eq:FSI_weak}
\left\{
\begin{aligned}
(\hat\rho_s\partial_{tt}\hat\bu_s,\hat\bphi)_{\hat\Omega_s}+\left(\rho_fD_t\bv_f,\bphi\right)_{\Omega_f}+(\tilde \bP_s,\nabla\hat\bphi)_{\hat\Omega_s}&+(\bsigma_f,\bepsilon(\bphi))_{\Omega_f}\\
&=\langle J\hat{g}_s,\hat\bphi\rangle+\langle g_f,\bphi\rangle,\\
(\nabla\cdot \bv_f,q)_{\Omega_f}&=0,\\
\bv_f\circ \bx_s&=\partial_t \hat\bu_s,\quad \text{on } \hat \Gamma.\\
\end{aligned}
\right.
\end{equation}
\end{description}

\begin{remark}
The solution $\bv_f$, $p$ and $\hat\bu_s$ are in some specific function spaces that require sufficient regularity in the time variable.  Since the regularity in time variable is not discussed in this paper,  we do not introduce these spaces in the weak formulation. 
\end{remark}

\section{Finite element discretization based on the ALE method}
In this section, we consider both time and space discretizations of
Equations (\ref{eq:FSI_weak}) and discuss the well-posedness. We first discretize the time variable $t$ with uniform time step size
$k=\Delta t$:
$$
t^n=nk, \quad n=0, 1, \ldots,
$$
and use the finite difference method to discretize time derivatives. For the space-time formulation of FSI, we refer to \cite{Tezduyar.T;Sathe.S2007a,Tezduyar.T;Sathe.S;Stein.K2006a} and references therein.

Since the function spaces usually depend on $t$, we use the superscript $n$ to indicate that the function space is at time $t^n.$ For example,
\begin{equation*}
 \label{spaceVn}
\mathbb{V}^n:=\{(\bv_f,\hat\bv_s)\in H^1_D(\Omega_f(t^n))\times H^1_D(\hat\Omega_s) \text{ such that }
\bv_f\circ\bx_s^n=\hat\bv_s, \text{ on }\hat\Gamma\}.
\end{equation*}

We use an ALE approach for the discretization of spatial variable.
In this approach, the structure domain is discretized by a fixed mesh
on the initial domain $\hat\Omega_s$ and the fluid domain is discretized by a sequence of moving meshes on the moving domain $\Omega_f(t)$.
\subsection{Time discretization}
\paragraph{Time discretization for the structure domain}

Without loss of generality,  we
consider for the time discretization of the structure variables the following simple finite difference schemes:
\begin{equation}
\label{eq:time_derivative_difference}
\begin{aligned}
(\partial_{t} \hat\bu_s)^{n+1}
\approx&
(\partial_{t,h} \hat\bu_s)^{n+1}
\equiv\frac{\hat\bu_s^{n+1}-\hat\bu_s^{n}}{k}, \\
 (\partial_{tt}\hat\bu_s)^{n+1}
\approx &(\partial_{tt,h}\hat\bu_s)^{n+1}
\equiv\frac{\hat\bu_s^{n+1}-2\hat\bu_s^{n}+\hat\bu_s^{n-1}}{k^2}.
\end{aligned}
\end{equation}
Other popular time discretization schemes such as the Newmark method \cite{Newmark.N1959a} can also be used.

\paragraph{Time discretization for the fluid domain by moving
  meshes} 
We need to find a mapping to move the fluid mesh such that it matches
the structure displacement on $\hat\Gamma$ and remains non-degenerate
in $\Omega_f$ as time evolves.  This mapping is a diffeomorphism in
continuous case, and we use piecewise polynomials to approximate it in
discrete case.  For a triangular mesh, only piecewise linear functions preserve the triangular shape of the elements in the mesh. In the rest of this paper, we assume that the mesh motion is piecewise linear.  We denote the image of $\hat\Omega_f$ under the piecewise linear map
$\bx_{h,f}$ by $\Omega_f^n$.  $\Omega_f^n$ is
discretized by a moving mesh with respect to time, denoted by
$T_h(\Omega_f^n)$.  Note that $\Omega_f^n$ is a polygonal domain in 2D, and a polyhedral domain in 3D. $\Omega_f^n$ is a result of numerical discretization, and is, in general, different from the domain shape
$\Omega_f(t^n)$ in the analytic solution of $(\ref{eq:FSI_weak})$. 

The technique we use to determine the mesh motion is the ALE method.  First introduced for finite
element discretizations of incompressible fluids in \cite{Hughes.T;Liu.W;Zimmermann.T1981a,Donea.J;Giuliani.S;Halleux.J1982a}, the ALE method provides an
approach to finding the fluid mesh that can fit the moving domain
$\Omega_f(t)$.

There are two main ingredients in the ALE approach:
\begin{enumerate}
\item Defining how the grid is moving with respect to time such that it matches
  the structure displacement at the fluid-structure interface.
\item Defining how the material derivatives are discretized on the moving grid. 
\end{enumerate}

Given the structure trajectory $\bx_s^n$ defined on $\hat\Gamma,$ the moving grid can be described by a diffeomorphism $\cA^n:\hat\Omega_f\mapsto \Omega_f$
that satisfies
\begin{equation}
\label{eq:ALE}
\left\{
\begin{aligned}
%-\Delta \cA&=0, &\text{ in }&\hat \Omega_f\\
\cA^n(\hat \bx)&=\hat \bx, &\text{ on }& \partial\hat\Omega_f\cap\partial\hat\Omega,\\
\cA^n(\hat \bx)&=\bx_s^n(\hat \bx,t), &\text{ on }& \hat \Gamma.\\
\end{aligned} 
\right.
\end{equation}

% figure to add
\begin{figure}
\begin{center}
\includegraphics[width=0.6\textwidth]{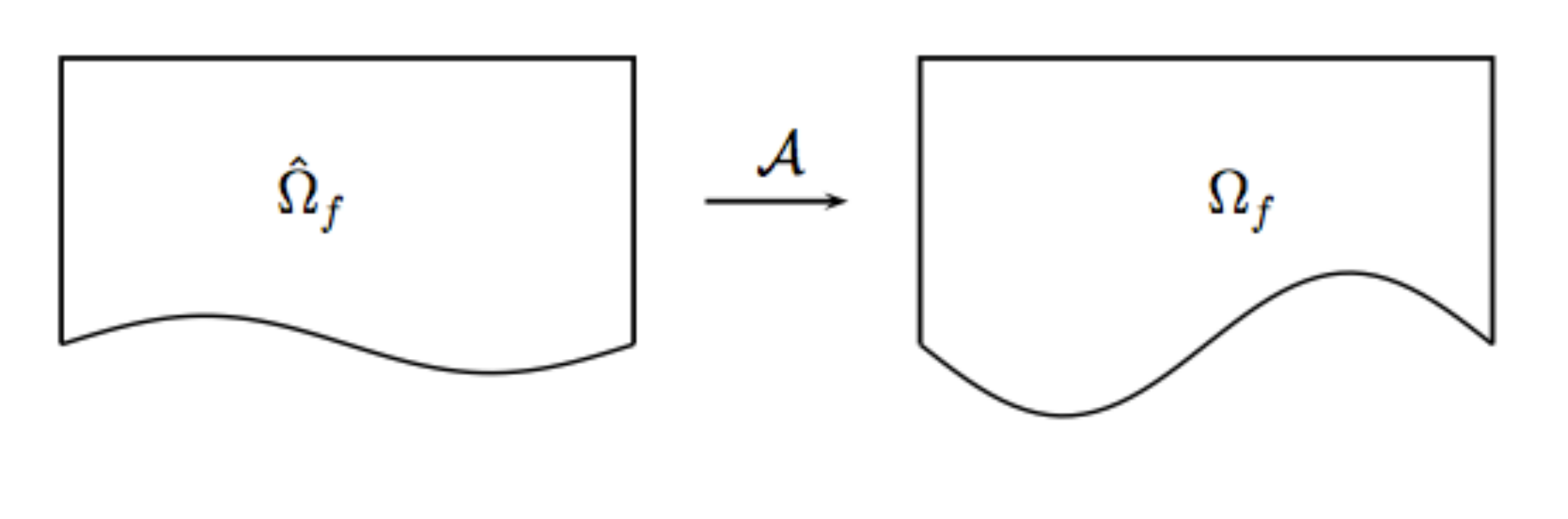}  
\end{center}
\caption{ALE mapping}
\end{figure}

%%  KY: the pstricks code for the picture

%\begin{figure}[h]
%\begin{center} 
% \begin{pspicture}*[0.5](0,0)(20,10)
%\psline[linewidth=1pt,linearc=0]{}(3,5)(3,3)(7,3)(7,5)
%\psplot[algebraic]{3}{7}{0.2*sin(3.14159*(x-3)/2)+5}
%\rput[br]{0}(5.5,3.8){$\hat{\Omega}_f$}
%
%\psline[linewidth=1pt,linearc=0]{->}(7.5,4)(8.5,4)
%\rput[br]{0}(8.2,4.2){$\mathcal{A}$}
%
%\psline[linewidth=1pt,linearc=0]{}(9,5)(9,3)(13,3)(13,5)
%\psplot[algebraic]{9}{13}{0.5*sin(-3.14159*(x-9)/2)+5}
%\rput[br]{0}(11.5,3.8){${\Omega}_f$}
%\end{pspicture}
%\end{center}
%\caption{ALE mapping}
%\end{figure}

ALE mappings satisfying (\ref{eq:ALE}) are by no means unique. In the
interior of $\hat \Omega_f$, the ALE mapping can be ``arbitrary".  One
popular approach to uniquely determine $\cA$ is to solve partial differential
equations
$$
\mathcal{L}\cA=0,\quad\mbox{in }\hat\Omega_f.
$$
A popular choice for the operator $\mathcal L$ is the Laplacian, $\mathcal{L}=-\Delta. $

To improve the quality of the fluid mesh with respect to the displacement
of the structure near the interface, the following elasticity model is
often used 
\cite{Donea.J;Huerta.A;Ponthot.J;Rodriguez-Ferran.A2004a}

$$
\mathcal{LA}=
-\mu\Delta\mathcal{A}-\lambda\nabla(\nabla\cdot\mathcal{A}).
$$
For more choices of formulating the ALE problem, we refer to \cite{Bazilevs.Y;Takizawa.K;Tezduyar.Ta,Donea.J;Giuliani.S;Halleux.J1982a} and references therein.

\paragraph{Discretization of the material derivative}
With the ALE mapping $\cA$ introduced, material derivatives can be written as follows

\begin{eqnarray*}
D_t\bv
&=&\partial_t\bv+(\bv\cdot\nabla)\bv\\
&=&\partial_t\bv+(\partial_t\cA\cdot\nabla)\bv+((\bv-\partial_t\cA)\cdot\nabla)\bv\\
&=&\partial_t\bv(\cA(\hat \bx,t),t))+((\bv-\partial_t\cA)\cdot\nabla)\bv.
\end{eqnarray*}
Using the simple approximation:
$$
\partial_t\bv(\cA(\hat \bx,t^{n+1}),t^{n+1}))
\approx \partial^\cA_{t,h}\bv|_{(\cA(\hat \bx,t^{n+1}),t^{n+1})}:=
\frac{v(\cA(\hat \bx, t^{n+1}), t^{n+1})-v(\cA(\hat \bx, t^{n}), t^{n})}{k}
$$
and 
$$
(\partial_t\cA)(\hat \bx,t) 
\approx (\partial_{t,h}\cA)(\hat \bx,t) 
:=\frac{\cA(\hat \bx, t^{n+1})-\cA(\hat \bx, t^{n})}{k},
$$
we obtain an approximation of material derivatives as follows:
\begin{equation}
  \label{eq:Dthv}
(D_t\bv)^{n+1}\approx (D_{t,h}\bv)^{n+1}:=
\partial^\cA_{t,h}\bv(\bx,t^{n+1})+((\bv-\partial_{t,h}\cA)\cdot\nabla)\bv(\bx,t^{n+1}),
\end{equation}
for $\bx=\cA(\hat \bx,t^{n+1})$.

% namely

%\footnote{\bf For higher order elements, degree of freedoms will have to be placed on non-grid points, we still need interpolation!!! {\color{red} KY: This is true if high order element is used for ALE. As long as we use P1 ALE mapping, we don't need to worry about it.}}
%\begin{equation}
%  \label{eq:Dthvi}
%  \begin{aligned}
%&(D_{t,h}\bv)^{n+1}(\bx_i^{n+1})\\
%=&\frac{\bv(\bx_i^{n+1}, t^{n+1})-\bv(\bx_i^{n},t^n)}{k}
%+\left(\left(\bv-\frac{\bx_i^{n+1}-\bx_i^n}{k}\right)\cdot\nabla\right)\bv(\bx_i^{n+1},t^{n+1}).
%\end{aligned}
%\end{equation}
 With the aforementioned discretization of material derivatives, we write the momentum equation of Navier-Stokes equations as
%$$\partial_t\bv+(\bv\cdot\nabla)\bv-\mu\Delta \bv+\nabla p=f$$
%can be rewritten in ALE form as follows
%we can just substitute the term $\partial \bv/\partial t$ with the Equation (\ref{eq:ALE_derivative}) to get the NS equations in ALE form
$$\rho_f\partial^\cA_{t,h}\bv_f+\rho_f((\bv_f-\partial_{t,h}\cA)\cdot\nabla)\bv_f-\mu\nabla\cdot\bsigma_f=g_f. $$

Once the time derivatives are discretized using (\ref{eq:time_derivative_difference}) and (\ref{eq:Dthv}), we obtain the fully implicit scheme.

\begin{description}
\item[Fully implicit (FI) scheme:] find $\bv^{n+1}_f\in \mathbb{V}_{f}^{n+1}$, $\hat\bu^{n+1}_s\in \hat {\mathbb V}_{s}$, $p\in\mathbb{Q}^{n+1}$ and $\cA^{n+1}\in H^1(\hat\Omega_f)$ such that for any $(\bphi,\hat\bphi)\in \mathbb{V}^{n+1}$ and $q\in \mathbb{Q}^{n+1}$,
\begin{equation}
\label{eq:fsi_fi}
\left\{
\begin{aligned}
( \hat\rho_s (\partial_{tt,h}\hat\bu_s)^{n+1},\hat\bphi)_{\hat\Omega_s}+( \rho_f(D_{t,h}\bv_f)^{n+1},&\bphi)_{\Omega_f}+(\bsigma_f^{n+1},\bepsilon(\bphi))_{\Omega_f}\\
+(\tilde\bP^{n+1}_s,\nabla\hat\bphi)_{\hat\Omega_s}&=\langle J\hat{g}_s,\hat\bphi\rangle+\langle g_f,\bphi\rangle,\\
(\nabla\cdot \bv^{n+1}_f,q)_{\Omega_f}&=0,\\
\bv_f^{n+1}\circ{\bx_{s}^{n+1}}&=(\partial_{t,h}\hat\bu_s)^{n+1},\quad\mbox{ on }\hat\Gamma,\\
\mathcal{L}\cA^{n+1}&=0,   ~~~~~~~~\qquad\text{ in } \hat \Omega_f,\\
\cA^{n+1}( \hat{\bx})&= \hat{\bx}, ~~~~~~~\qquad\text{ on } \partial\hat\Omega_f\cap\partial\hat\Omega,\\
\cA^{n+1}( \hat{\bx})&= \hat\bx+\hat\bu_s^{n+1}, ~~~~\text{ on } \hat \Gamma,\\
\end{aligned}
\right.
\end{equation}
\end{description}
The structure displacement $\hat\bu_s^{n+1}$ serves as the boundary condition for the ALE problem.  Note that $\cA^{n+1}$ has to be a homeomorphism. The fluid stress $\bsigma_f^{n+1}$ is defined by (\ref{eq:constitutive_f}) in terms of $\bv_f^{n+1}$ and $p^{n+1}$. The structure stress $\tilde \bP_s^{n+1}$ is defined by (\ref{eq:constitutive_s}) in terms of $\hat\bu_s^{n+1}$.

In the FI scheme, nonlinearity comes from the convection term and the dependence of the Navier-Stokes (NS) equations on the ALE mapping. To solve (\ref{eq:fsi_fi}), Newton's method or fixed-point iteration may be used to linearize the problem. %we do not talk about GE scheme here
%\item {\bf Geometry explicit  (GE)scheme}: The fluid mesh motion $\cA$ is explicit. The convection term is linearized by
%$$(\bv_f^{n+1}-\bomega^{n+1})\cdot \nabla\bv_f^{n+1}\approx (\bv_f^{n+1}-\bomega^{n+1})\cdot\nabla\bv_f^n+(\bv_f^{n}-\bomega^{n})\cdot\nabla\bv_f^{n+1}-(\bv_f^{n}-\bomega^n)\cdot\nabla\bv_f^n.$$

Another frequently used linearization of the FI scheme is the following geometry-convective explicit scheme\cite{Crosetto.P;Deparis.S;Fourestey.G;Quarteroni.A2010a,Crosetto.P2011a,Malossi.A;Blanco.P;Crosetto.P;Deparis.S;Quarteroni.A2013a}

\begin{description}
\item[ Geometry-convective explicit (GCE) scheme:] Find $\bv^{n+1}_f\in H_D^1(\Omega_f(t^{n}))$, $\hat\bu^{n+1}_s\in H_D^1(\hat\Omega_s)$, $p\in L^2(\Omega_f(t^{n}))$ and $\cA^{n+1}\in H^1(\hat\Omega_f)$ such that for any $(\bphi,\hat\bphi)\in \mathbb{V}^{n}$ and $q\in \mathbb{Q}^{n}$,
\begin{equation}
\footnotesize
\label{eq:fsi_gce}
\left\{
\begin{aligned}
( \rho_f(\partial_{t,h}^{\cA}\bv_f)^{n+1},\bphi )_{\Omega_f}+( \hat\rho_s & (\partial_{tt,h}\hat\bu_s)^{n+1},\hat\bphi )_{\hat\Omega_s}+(\bsigma_f^{n+1},\bepsilon(\bphi))_{\Omega_f}\\
+(\tilde\bP_s^{n+1},\nabla (\hat\bphi))_{\hat\Omega_s} &=\langle g_f+((\bv_f^{n}-\partial_{t,h}\cA^{n+1})\cdot \nabla) \bv_f^{n},\bphi\rangle_{\Omega_{f}}+\langle J\hat{g}_s,\hat\bphi\rangle_{\hat\Omega_s}, \\
(\nabla\cdot \bv^{n+1}_f,q)_{\Omega_f}&=0,\\
\bv_f^{n+1}\circ{\bx_{h}^{n}}&=(\partial_{t,h}\hat\bu_s)^{n+1},\quad\qquad\qquad\qquad\mbox{ on }\hat\Gamma,\\
\mathcal{L}\cA^{n+1}&=0,   ~~~~~~~~~~~\qquad\qquad\qquad\qquad\text{ in } \hat \Omega_f,\\
\cA^{n+1}( \hat{\bx})&= \hat{\bx}, ~~~~~~~~~~~\qquad\qquad\qquad\qquad\text{ on } \partial\hat\Omega_f\cap\partial\hat\Omega,\\
\cA^{n+1}( \hat{\bx})&= \hat\bx+\hat\bu_s^{n}(\hat\bx)+k\bv_f^{n}\circ\bx_h^n(\hat\bx), ~~~~~\text{ on } \hat \Gamma.\\
\end{aligned}
\right.
\end{equation}
\end{description}
The boundary condition for $\cA^{n+1}$ is given by $\hat\bu_s^n$, the structure displacement, and $\bv_f^n$, the fluid velocity,  from the previous time step. Thus, the solution of $\cA^{n+1}$ is decoupled from solving momentum and continuity equations.  After $\cA^{n+1}$ is solved, the mapping from $\hat\Omega_f$ to $\Omega_f(t^n)$ is known and $\partial_{t,h}\cA^{n+1}$ can be calculated.  In (\ref{eq:fsi_gce}), the convection term is explicitly calculated using $\partial_{t,h}\cA^{n+1}$ and $\bv_f^{n}$
\begin{equation}
\label{eq:explicit_convection}
(\bv_f^{n+1}-\partial_t\cA^{n+1})\cdot \nabla\bv_f^{n+1}\approx (\bv_f^{n}-\partial_{t,h}\cA^{n+1})\cdot \nabla\bv_f^{n}.
\end{equation}

The GCE scheme in the literature has the following linearization of the convection term \cite{Crosetto.P;Deparis.S;Fourestey.G;Quarteroni.A2010a,Crosetto.P2011a,Malossi.A;Blanco.P;Crosetto.P;Deparis.S;Quarteroni.A2013a}:
\begin{equation}
\label{eq:semiexplicit_convection}
(\bv_f^{n+1}-\partial_t\cA^{n+1})\cdot \nabla\bv_f^{n+1}\approx (\bv_f^{n}-\partial_{t,h}\cA^{n+1})\cdot \nabla\bv_f^{n+1}.
\end{equation}
We take (\ref{eq:explicit_convection}) instead of (\ref{eq:semiexplicit_convection}) since the former results in symmetric variational problems and facilitates our analysis. However, we also briefly discuss about the unsymmetric cases due to (\ref{eq:semiexplicit_convection}) in the next section.

Since the solution of $\cA^{n+1}$ is decoupled from momentum and continuity equations,  we do not rewrite the equations about $\cA$ in the GCE scheme in the rest of the paper.

\subsubsection*{Change of variables for structure equations}
Note that the discretized interface condition for the velocity is
$$\bv_f^n\circ\bx_{s,h}^n=\frac{\hat\bu_s^n-\hat\bu_s^{n-1}}{\Delta t},\quad\mbox{ on }\hat\Gamma.$$
The velocities of fluid and structure are assumed to be continuous on the interface $\hat \Gamma$. By introducing the structure velocity in the same fashion as in (\ref{eq:time_derivative_difference}),
\begin{equation}
\label{eq:change_v}
\hat\bv_s^n=\frac{\hat\bu_s^n-\hat\bu_s^{n-1}}{\Delta t},
\end{equation}
 the interface condition becomes
$$\bv_f^n\circ\bx_{s}^n=\hat\bv_s^n,\quad\mbox{ on }\hat\Gamma.$$

Therefore, the unknowns $\bv_f$ and $\hat\bv_s$ are continuous on $\Gamma$ with a change of coordinates for $\bv_f$ and $(\bv_f^n,\hat\bv_s^n)$ belongs to the space $\mathbb{V}^n$.  Instead of $\hat\bu_s$, we take $\hat\bv_s$ as one of the unknowns since it facilitates our theoretical analysis in the next section.  We change the variables in the GCE scheme and get the modified GCE scheme:

\begin{description}
\item[ Modified GCE scheme:] Find $(\bv_f^{n+1},\hat\bv_s^{n+1})\in \mathbb{V}^{n}$ and $p\in\mathbb{Q}^{n}$ such that $\forall(\bphi,\hat\bphi)\in \mathbb{V}^{n}$ and $\forall q\in \mathbb{Q}^{n}$,

\begin{equation}
\label{eq:fsi_mgce}
\left\{
\begin{aligned}
\frac{1}{k} ( \rho_f\bv_f^{n+1},\bphi )_{\Omega_f}+\frac{1}{k} ( \hat\rho_s \hat\bv_s^{n+1},\hat\bphi )_{\hat\Omega_s}+&(\bsigma_f^{n+1},\bepsilon(\bphi))_{\Omega_f}\\
+k(\tilde\bP_s(\hat\bv^{n+1}_s),\nabla \hat\bphi)_{\hat\Omega_s} &=\langle \tilde g_f,\bphi\rangle_{\Omega_f}+\langle \tilde{g}_s,\hat\bphi\rangle_{\hat\Omega_s}, \\
(\nabla\cdot \bv^{n+1}_f,q)_{\Omega_f}&=0,\\
\end{aligned}
\right.
\end{equation}
\end{description}
where 
$$ \tilde g_f=g_f+((\bv_f^{n}-\partial_{t,h}\cA^{n+1})\cdot \nabla) \bv_f^{n}+\rho_f\bv_f^n/k$$
$$\tilde{g}_s=J\hat{g}_s+\hat\rho_s\hat\bv_s^n/k-\tilde\bP_s(\hat\bu_s^n).$$
  $\tilde\bP_s(\hat\bv_s^{n+1})$ is in terms of $\hat\bv_s^{n+1}$ instead of $\hat\bu_s^{n+1}$; that is,
$$\tilde\bP_s(\hat\bv_s^{n+1})=\mu_s\bepsilon(\hat\bv_s^{n+1})+\lambda_s\nabla\cdot\hat\bv_s^{n+1}\bI.$$

\subsection{Space discretization}
The structure domain $\hat\Omega_s$ is discretized by a fixed triangulation, denoted by $T_h(\hat\Omega_s)$.  The
corresponding finite element space is defined as:
$$
\hat {\mathbb{V}}_{h,s}=\{ \hat\bu \in H_D^1(\hat\Omega_s): \hat\bu|_{\tau}\in
\mathcal{P}_m,\forall \tau\in T_{h}(\hat\Omega_s)\}.
$$

The fluid domain $\Omega_f$ is moving over time due to the interaction. At time $t=0$, we have the initial triangulation $T_h(\hat\Omega_f)$ on $\hat\Omega_f$. In this paper we only consider the case in which $T_h(\hat\Omega_s)$ and $T_h(\hat\Omega_f)$ are matching on the interface $\hat\Gamma$. 

For $t>0$, the fluid domain $\Omega_f(t)$ evolves due to the motion of interface. Therefore, we discuss the discrete interface motion first. The structure displacement $\bu_s$ provides the motion of the interface. Note that $\bu_s$ is in some finite element space and, therefore, the displacement of the interface $\Gamma$ is piecewise polynomial.  This approximation of interface motion introduces additional error, besides that of approximating velocity in $H^1$ and  pressure in $L^2$ with piecewise polynomials.  Since only the triangular elements are considered in this paper, we use piecewise linear interface motion, which transforms a triangular element to another triangular element.  If higher order elements are used for the structure displacement, like P2, interpolations have to be performed in order to get P1 interface motion. For example, the interface motion of GCE scheme is approximated by
%$$\bx_s^{n+1}(\hat\bx)\approx\hat\bx+\Pi_h^1\hat\bu_s^{n}(\hat\bx),\quad\hat\bx\in\hat\Gamma,$$
%or, in terms of $\hat\bv_s^{n}$ and $\hat\bu_s^{n-1}$,
$$\bx_s^{n+1}(\hat\bx)\approx\hat\bx+\Pi_h^1(\hat\bu_s^{n}+k\bv_f^{n}\circ\bx_h^n)(\hat\bx),\quad\hat\bx\in\hat\Gamma. $$
%$$\hat\bx+\hat\bu_s^{n}(\hat\bx)+k\bv_f^{n}\circ\bx_h^n(\hat\bx)$$
Here, $\Pi_h^1$ is a interpolation operator, the range of which is the space of the continuous and piecewise linear functions.  

%\paragraph{Explicit approximation of the interface}
%Finite element discretization of mesh motion introduces error. The trajectory of structure has the following Taylor expansion
%\begin{eqnarray*}
%\bx_s(\hat \bx, t^{n+1}) 
%&=&\bx_s(\hat \bx, t^{n}) + k\frac{\partial \bx_s}{\partial t}(\hat \bx, t^{n})
%+{\cal O}(k^2) \\
%&=&  \bx_s(\hat \bx,t^n) +k\hat\bv_s(\hat \bx, t^{n}) +{\cal O}(k^2). \\
%%&=& \hat \bx+(1+ k)\bv_f(\hat \bx, t^{n}) +{\cal O}(k^2) 
%\end{eqnarray*}
%
%In the discretized formulation, we define a discrete flow map
%$\bx_{h,s}^{n+1}$ as follows:
%\begin{equation}
%\label{xsh}
%\bx_{h,s}^{n+1}(\hat \bx)=\bx_{h,s}^{n}(\hat \bx)+k(\Pi_h^1\hat\bv_f^n)(\hat \bx), \quad \hat \bx\in \hat\Gamma.   
%\end{equation}
%Here $\Pi_h^1$ is a linear interpolation operator.  We use $\Pi_h^1$ because the mesh moving function is piecewise linear and $\bv_f$ may be in some higher order finite element space.  The above
%definition of approximate flow map is very special.  For a more general
%treatment, a Lagrange multiplier needs to be introduced, see
%\cite{Crosetto.P;Deparis.S;Fourestey.G;Quarteroni.A2010a,Crosetto.P;Reymond.P;Deparis.S;Kontaxakis.D;Stergiopulos.N;Quarteroni.A2011a}, but this will not be considered in this paper.

\paragraph{Discrete ALE problem} With the discrete boundary motion provided, we solve a discrete version of the ALE equations.  We only consider piecewise linear ALE mappings to keep the mesh triangular.  Once we obtain the discrete ALE mapping $\cA_h$, the fluid triangulation on the current configuration can be obtained.  Denote the set of grid points for the triangulation of $T_h(\hat\Omega_f) $ by
$$
\hat {\cal N}_h=\{\hat \bx_i; i=1:n_h\}. 
$$ 
Then, the set of grid points for the triangulation of
$T_h(\Omega_f^n) $ is given by
\begin{equation*}\label{Nh}
{\cal N}_h^n=\{x_i^n=\cA_h(\hat \bx_i, t^n)| i=1:n_h, \hat x_i\in {\mathcal{\hat N}}_h\}.   
\end{equation*}
Therefore, $T_h(\Omega_f^n)$ is obtained accordingly.  Since the grid points are moved according to $\cA_h$, we know that no interpolation is
needed for evaluating the material derivative $D_t\bv$ at grid points.

We define the finite element spaces for the fluid velocity and pressure on the triangulation $T_h(\Omega_f^n)$:
$$
\mathbb{V}_{h,f}^{n}=\{ \bv \in H_D^1(\Omega_f^n): \bv|_{\tau}\in
\mathcal{P}_m,\forall \tau\in T_h(\Omega_f^n)\},
$$
and
$$
\mathbb{Q}_{h}^n=\{q\in L^2(\Omega_f^n): q|_{\tau}\in
\mathcal{P}_l,\forall\tau\in T_h(\Omega_f^n)\},
$$
where $m$ and $l$ denote the orders of finite elements.
 
%\footnote{Kai: There seem to be a lot of  problems here.  Without using Lagrangian multiplier on the interface,  linear interpolation has to be used in \eqref{xsh},  which would limit the accuracy of the scheme regardless the order of finite element space $k$. {\color{red} KY: I agree. The accuracy will be limited for triangular mesh. }If $k>1$, do we also need use isoparametric elements to approximate $\hat\Gamma$? {\color{red} KY: I think we can if we want to improve the overall accuracy. }}

\paragraph{Global finite element space}
We define the finite element approximation of \eqref{spaceV} as
follows:
$$
\mathbb{V}_h^{n+1}:=\{(\bv_f,\hat \bv_s): \bv_f\in
\mathbb{V}_{h,f}^{n+1},~~ \hat \bv_s\in \hat {\mathbb V}_{h,s},~~
\bv_f\circ \bx_{h,s}^{n+1}=\hat\bv_s, \text{ on }\hat\Gamma\}.
$$
Note that the space is for both velocity unknowns and the test functions in the variational problem.

\begin{description}
\item[ Modified GCE finite element scheme:] Find $(\bv_f^{n+1},\hat\bv_s^{n+1})\in \mathbb{V}_h^{n}$ and $p\in\mathbb{Q}_h^{n}$ such that for all $(\bphi,\hat\bphi)\in \mathbb{V}_h^{n}$ and $q\in \mathbb{Q}_h^{n}$,

\begin{equation}
\label{eq:fsi_mgce_fem}
\left\{
\begin{aligned}
\frac{1}{k} ( \rho_f\bv_f^{n+1},\bphi )_{\Omega_f}+\frac{1}{k} ( \hat\rho_s \hat\bv_s^{n+1},\hat\bphi )_{\hat\Omega_s}+&(\bsigma_f^{n+1},\bepsilon(\bphi))_{\Omega_f}\\
+k(\tilde\bP_s(\hat\bv^{n+1}_s),\nabla  \hat\bphi)_{\hat\Omega_s} &=\langle \tilde g_f,\bphi\rangle_{\Omega_f}+\langle \tilde{g}_s,\hat\bphi\rangle_{\hat\Omega_s}, \\
(\nabla\cdot \bv^{n+1}_f,q)_{\Omega_f}&=0,\\
\end{aligned}
\right.
\end{equation}
\end{description}

\begin{remark}
\begin{itemize}
\item GCE can be used not only in weakly coupled explicit algorithms for FSI, but also in fixed-point iteration to achieve strong coupling.

Newton's method can also be used to linearize the FI scheme \cite{Fernandez.M;Moubachir.M2005a}, where shape derivatives have to be calculated.  We do not consider this type of discretization in this paper.
\item
There are many different approaches to enforce interface conditions. Many of them use Lagrange multipliers \cite{Crosetto.P;Deparis.S;Fourestey.G;Quarteroni.A2010a,Crosetto.P;Reymond.P;Deparis.S;Kontaxakis.D;Stergiopulos.N;Quarteroni.A2011a} and this introduces additional degrees of freedom.  An approach to avoiding Lagrange multipliers is to consider velocity and displacement in the entire domain \cite{Turek.S;Hron.J2006a,Hron.J;Turek.S2006a,Dunne.T;Rannacher.R;Richter.T2010a}. The velocity in the structure domain is naturally the time derivative of structure displacement, while the displacement in the fluid domain is the mesh displacement \cite{Hron.J;Turek.S2006a}. In \cite{Quaini.A;Quarteroni.A2007a,Badia.S;Quaini.a;Quarteroni.a2008b,Badia.S;Quaini.A;Quarteroni.A2008a}, fluid velocity, pressure, and structure velocity are considered as unknowns.  In our approach, we also use this velocity-pressure formulation of FSI to facilitate our analysis.
\end{itemize}
\end{remark}

In the next section, we start our theoretical analysis based on the formulation in (\ref{eq:fsi_mgce}) and (\ref{eq:fsi_mgce_fem}).

\subsection{Reformulation as a saddle point problem}

 For brevity, we do not keep the superscript $n$ and we use $\mathbb{V}_h$ and $\mathbb{Q}_h$ instead of $\mathbb{V}_h^n$ and $\mathbb{Q}_h^n$.  In this section, we focus on the linear systems resulting from (\ref{eq:fsi_mgce}) and formulate them as saddle point problems.  For the space $\mathbb{V}$,  we assume that $\bx_s=\bx_{h,s}$; namely, $\bx_s$ in the definition of  $\mathbb{V}$ is assumed to be piecewise linear on the triangulation $T_h(\hat\Omega_s)$. As a consequence, $\mathbb{V}_h$ is a subspace of $\mathbb{V}$.  Similarly, $\mathbb{Q}_h\subset\mathbb{Q}$.   For $\bv\in \mathbb{V}$, we use $\bv_f$ and $\hat\bv_s$ to denote its fluid and structure components, respectively. This convention applies to other functions in $\mathbb{V}$, such as $\bu=(\bu_f,\hat\bu_s)\in \mathbb{V}$ and $\bphi=(\bphi_f,\hat\bphi_s)\in \mathbb{V}$. To guarantee the continuity of velocity on interface, we use polynomials of the same order for the fluid velocity and structure velocity.

  We introduce the following definition of the $H^1$ norm for $\bv=(\bv_f,\hat\bv_s)\in \mathbb{V}$:
 $$\|\bv\|_1^2=\|\bv_f\|_{1,\Omega_f}^2+\|\hat\bv_s\|_{1,\hat\Omega_s}^2,$$
and define the following bilinear forms for $\bv=(\bv_f,\hat\bv_s)\in\mathbb{V}$, $\bphi=(\bphi_f,\hat\bphi_s)\in \mathbb{V}$ and $p\in \mathbb{Q}$

\begin{equation*}
\begin{aligned}
a(\bv,\bphi)=&\frac{1}{k}(\rho_f\bv_f,\bphi_f)_{\Omega_f}+\frac{1}{k}(\hat\rho_s\hat\bv_s,\hat\bphi_s)_{\hat\Omega_s}+(\mu_f\bepsilon (\bv_f),\bepsilon ( \bphi_f))_{\Omega_f}\\
&+k(\mu_s\bepsilon (\hat\bv_s),\bepsilon(\hat\bphi_s))_{\hat\Omega_s}+k (\lambda_s\nabla\cdot \hat\bv_s,\nabla\cdot \hat\bphi_s)_{\hat\Omega_s}
\end{aligned}
\end{equation*}
 and
$$b(\bv,p)=(\nabla\cdot \bv_f,p)_{\Omega_f}.$$

 In this paper, we assume the material parameters to be constant within the fluid domain and the structure domain.

With the bilinear forms defined, (\ref{eq:fsi_mgce}) can be reformulated as a saddle point problem:

\begin{description}
\item[]
Find $\bv\in \mathbb{V}$ and  $p\in \mathbb{Q}$  such that
\begin{equation}
\label{eq:saddle}
\left\{
\begin{aligned}
&a(\bv,\bphi)+b(\bphi,p)&=&\langle \tilde g,\bphi\rangle,&\forall& \bphi\in \mathbb{V},\\
&b(\bv,q) &=&0,&\forall& q\in \mathbb{Q},\\
\end{aligned}
\right.
\end{equation}
\end{description}
where $\langle\tilde g,\bphi\rangle=\langle \tilde g_f,\bphi_f\rangle+\langle \tilde g_s,\hat\bphi_s\rangle.$
This type of problems has various applications, for example in Stokes equations and constrained optimization, and is well studied \cite{Brezzi.F;Fortin.M1991a,Girault.V;Raviart.P1986a}.

In order to study the well-posedness of this problem, we need to carefully define norms for $\mathbb{V}$ and $\mathbb{Q}$ as

\begin{equation*}
\begin{aligned}
\mbox{for all }\bv\in \mathbb{V},~~\|\bv\|_V^2&:=a(\bv,\bv)+r\|\nabla\cdot\bv_f\|_{0,\Omega_f}^2,\\
\mbox{for all } q\in\mathbb{Q},~~\|q\|^2_Q&:=r^{-1}\|q\|^2_0,\\
\end{aligned}
\end{equation*}
where 
\begin{equation}
\label{eq:r}
 r=\max\{1,\mu_f, \rho_fk^{-1}, \hat\rho_sk^{-1}, k\mu_s,k\lambda_s\}.
 \end{equation}

It is well known that (\ref{eq:saddle}) is well-posed if the following conditions can be verified \cite{Girault.V;Raviart.P1986a}

\begin{itemize}
\item 
\begin{equation}
\label{eq:brezzi_a}
a(\cdot,\cdot) \mbox{ is bounded and coercive in } \mathbb{Z}:=\{\bv\in\mathbb{V}| \nabla\cdot\bv=0\mbox{ in }\Omega_f\},
\end{equation}
 \item 
 \begin{equation}
 \label{eq:brezzi_b}
 \begin{aligned}
 b(\cdot,\cdot)  \mbox{ is bounded }&\mbox{and satisfies the inf-sup condition } \\
 &\inf_{p\in \mathbb{Q}}\sup_{\bv\in \mathbb{V}}\frac{b(\bv,p)}{\|\bv\|_V\|p\|_Q}\geq \beta>0.\\
 \end{aligned}
 \end{equation}
\end{itemize}
In the rest of the paper, we prove the boundedness and coercivity of $a(\cdot,\cdot)$ and the inf-sup condition of $b(\cdot,\cdot)$ in order to show the well-posedness of saddle point problems, like (\ref{eq:saddle}).

By definition, it is straightforward to prove the conditions on $a(\cdot,\cdot)$ since
 \begin{equation}
 \label{eq:a_coercive}
 a(\bv,\bv)= \|\bv\|_V^2,\quad\forall\bv\in\mathbb{Z}.
 \end{equation}

%We add the term $r\|\nabla\cdot\bu\|_{\Omega_f}^2$ in the norm $\|\cdot\|_V$ and the scaling factor $r^{1/2}$ in the norm $\|\cdot\|_Q$ such that $b(\bu,q)$ is uniformly bounded with respect to parameters
The boundedness of $b(\cdot,\cdot)$ follows from the definition:
 \begin{equation}
 \label{eq:b_bounded}
 b(\bv,q)\leq \|\nabla\cdot\bv\|_{0,\Omega_f}\|q\|_0\leq  r^{1/2}\|\nabla\cdot\bv\|_{0,\Omega_f}r^{-1/2}\|q\|_0 \leq\|\bv\|_V\|q\|_Q.
  \end{equation}

Now, we need to prove the inf-sup condition of $b(\cdot,\cdot)$.   First, we have the following lemma.

\begin{lemma}[\cite{Bramble.J;Lazarov.R;Pasciak.J2001a}]
\label{lm:stokes_infsup}
Let $\partial\Omega_D\subset\partial\Omega$ satisfy  $|\partial\Omega_D|>0$ and $|\partial\Omega\setminus\partial\Omega_D|>0$. Then there exists a constant $C$ such that
$$\sup_{\bv\in H_D^1(\Omega)}\frac{(\nabla\cdot\bv,q)}{\|\bv\|_{1,\Omega}}\geq C \|q\|_{0,\Omega},\quad\mbox{for all } q\in L^2(\Omega),$$
where $H_D^1(\Omega)=\{\bv\in H^1(\Omega)| \bv(\bx)=0,~~\mbox{ for all }\bx\in\partial\Omega_D\}.$
\end{lemma}

The following lemma is the key ingredient in proving the well-posedness of (\ref{eq:saddle}).  In this case, the fluid domain is deformed due to the motion of the structure.  In the  GCE scheme, $\bx_s$ is treated explicitly and the inf-sup constant depends on $\bx_s$.

\begin{lemma}
\label{lm:infsup_b_ale}
Assume that $$\bx_s\in W^{1,\infty}(\hat\Omega_s)\quad\mbox{and} \quad \inf_{\hat\bx\in \hat\Omega_s}\det (\nabla\bx_s(\hat\bx)) >0.$$
Then the following inf-sup condition holds
$$
\inf_{q\in\mathbb{Q}}\sup_{\bv\in\mathbb{V}}\frac{b(\bv,q)}{\|\bv\|_1\|q\|_0}\gtrsim \frac{1}{d_0^{N/2+1}d_1},
$$
where  
\begin{equation}
\label{eq:d0d1}
d_0=\max\left\{\sup_{\hat\bx\in \hat\Gamma}\|\nabla\bx_s(\hat\bx)\|_{2},1\right\},\quad d_1=\max\left\{\sup_{\hat\bx\in \hat\Gamma}\left\{\det(\nabla\bx_s(\hat\bx))^{-1}\right\}, 1\right\}.
\end{equation}
\end{lemma}
Note that $N=2,3$ is the dimension of the FSI problem and  $\|\nabla\bx_s\|_2$ is the induced matrix 2-norm.
\begin{proof}

Given $q\in \mathbb{Q}=L^2(\Omega_f)$, we can find $\bv_f\in H_D^1(\Omega_f)=\{\bv\in H^1(\Omega_f) \mbox{ and } \bv_f|_{\partial\Omega_f\cap\partial\Omega}=0\}$ such that 

$$\frac{(\nabla\cdot\bv_f,q)_{\Omega_f}}{\|\bv_f\|_{1,\Omega_f}\|q\|_0}\gtrsim 1.$$

Then, we take $\hat\bv_s\in\hat{\mathbb{V}}_{h,s}$ satisfying $\bv_f\circ \bx_s=\hat\bv_s$ on $ \hat\Gamma$ and
\begin{equation}
\label{eq:harmonic_extension}
\int_{\hat\Omega_s}\nabla\hat\bv_s:\nabla\bphi=0,\quad\mbox{ for all } \bphi\in H_0^1(\hat\Omega_s).
\end{equation}

Then, we know that $\bv:=(\bv_f,\hat\bv_s)\in \mathbb{V}_h$ and $\|\hat\bv_s\|_{1,\hat\Omega_s}\lesssim \|\hat\bv_s\|_{1/2,\partial\hat\Omega_s}$.

% let $\bv\in \mathbb{V}$  s.t.
%
%$$
%\bv=\left\{
%\begin{aligned}
%&\bv_f & \mbox{ in } \Omega_f,\\
%&\bv_f\circ \bx_s=\hat\bv_s &\mbox{ on } \hat\Gamma,\\
%&\hat \bv_s &\mbox{ in } \hat \Omega_s,\\
%\end{aligned}
%\right.
%$$
%and $\hat \bv_s$ satisfies that
%$$\int_{\hat\Omega_s}\nabla\hat\bv_s:\nabla\bphi=0,\quad\forall \bphi\in H_0^1(\hat\Omega_s).$$
%
%Define $d_0=\max\{\|\nabla\bx_s\|_{\infty,\hat\Gamma},1\}$, $d_1=\max\{\|\det(\nabla\bx_s)^{-1}\|_{\infty,\hat\Gamma}, 1\}.$

The structure flow map $\bx_s$ maps from $\hat\Gamma$ to $\Gamma$. By Nanson's formula \cite{Bazilevs.Y;Takizawa.K;Tezduyar.Ta}, the following inequality about surface elements $ds$ and $d\hat s$ holds
$$ds(\bx_s(\hat\bx))\leq \det(\nabla\bx_s)\|(\nabla\bx_s)^{-1}\|_2d\hat s(\hat \bx).$$

Given $\bx,\by\in \Gamma$, $|\bx-\by|$ denotes the distance between $\bx$ and $\by$ on $\Gamma$.  It is easy to verify that

$$|\bx_s(\hat\bx)-\bx_s(\hat\by)|\leq \sup_{\bz\in\Gamma}\|\nabla\bx_s(\bz)\|_{2}|\bx-\by|\leq d_0|\bx-\by|$$
and, accordingly,
$$\mbox{dist}(\bx_s(\hat\bx),\Gamma)=\inf_{\by\in \Gamma}|\bx_s(\hat\bx)-\by|=\inf_{\hat \by\in\hat \Gamma}|\bx_s(\hat\bx)-\bx_s(\hat\by)|\leq d_0\inf_{\hat\by\in \hat \Gamma}|\hat\bx-\hat\by|=d_0\mbox{dist}(\hat\bx,\hat \Gamma).$$

The integral on the interface $\hat\Gamma$ can be estimated as follows
\begin{equation*}
\begin{aligned}
&|\bv_f\circ \bx_s|^2_{H_{00}^{1/2}(\hat\Gamma)}\\
=&\int_{\hat\Gamma}\int_{\hat\Gamma}\frac{|\bv_f\circ\bx_s(\hat \bx)-\bv_f\circ\bx_s(\hat \by)|^2}{|\hat\bx-\hat\by|^N}d\hat s(\hat\bx)d\hat s(\hat\by)+\int_{\hat\Gamma}\frac{|\bv_f\circ\bx_s(\hat\bx)|^2}{\mbox{dist}(\hat\bx,\partial\hat\Gamma)}d\hat s(\hat\bx)\\
=&\int_{\hat\Gamma}\int_{\hat\Gamma}\frac{|\bv_f\circ\bx_s(\hat \bx)-\bv_f\circ\bx_s(\hat\by)|^2}{|\bx_s(\hat\bx)-\bx_s(\hat\by)|^N}\frac{|\bx_s(\hat\bx)-\bx_s(\hat\by)|^N}{|\hat\bx-
\hat\by|^N}d\hat s(\hat\bx)ds(\hat\by)\\
&+\int_{\hat\Gamma}\frac{|\bv_f\circ\bx_s(\hat\bx)|^2}{\mbox{dist}(\bx_s(\hat\bx),\partial\Gamma)}\frac{\mbox{dist}(\bx_s(\hat\bx),\partial\Gamma)}{\mbox{dist}(\hat\bx,\partial\hat\Gamma)}d\hat s(\hat\bx)\\
\leq&d_0^N\int_{\hat\Gamma}\int_{\hat\Gamma}\frac{|\bv_f\circ\bx_s(\hat\bx)-\bv_f\circ\bx_s(\hat\by)|^2}{|\bx_s(\hat\bx)-\bx_s(\hat\by)|^N}d\hat s(\hat\bx)d\hat s(\hat\by)+d_0\int_{\hat\Gamma}\frac{|\bv_f\circ\bx_s(\hat\bx)|^2}{\mbox{dist}(\bx_s(\hat\bx),\partial\Gamma)}d\hat s(\hat\bx)\\
\leq&d_0^N\int_{\Gamma}\int_{\Gamma}\frac{|\bv_f(\bx)-\bv_f(\by)|^2}{|\bx-\by|^N}\det(\nabla\bx_s)^{-2}\|\nabla\bx_s\|_2^{2}ds(\bx)ds(\by)\\
&+d_0\int_{\Gamma}\frac{|\bv_f(\bx)|^2}{\mbox{dist}(\bx,\partial \Gamma)}\det(\nabla\bx_s)^{-1}\|\nabla\bx_s\|_2ds(\bx)\\
\leq& d_0^{N+2}d_1^2\int_{\Gamma}\int_{\Gamma}\frac{|\bv_f(\bx)-\bv_f(\by)|^2}{|\bx-\by|^N}ds(\bx)ds(\by)+ d_0^{2}d_1\int_{\Gamma}\frac{|\bv_f(\bx)|^2}{\mbox{dist}(\bx,\partial\Gamma)}ds(\bx)\\
\end{aligned}
\end{equation*}
and
$$\|\bv_f\circ\bx_s\|_{L^2(\hat\Gamma)}^2\leq d_0d_1 \|\bv_f\|^2_{L^2(\Gamma)}.$$
Therefore, 
$$\|\bv_f\circ\bx_s\|_{H_{00}^{1/2}(\hat\Gamma)}^2\leq d_0^{N+2}d_1^2 \|\bv_f\|^2_{H_{00}^{1/2}(\Gamma)}.$$

Based on the intrinsic definition of the semi norm
$$|\bv_f|^2_{H_{00}^{1/2}(\Gamma)}= \int_{\Gamma}\int_{\Gamma}\frac{|\bv_f(\bx)-\bv_f(\by)|^2}{|\bx-\by|^n}ds(\bx)ds(\by)+\int_{\Gamma}\frac{|\bv_f|^2}{\mbox{dist}(x,\partial\Gamma)}ds(x),$$
we know that \cite{Xu.J;Zou.J1998a}
$$| \bv_f\circ\bx_s|_{1/2,\partial\hat\Omega_f}\ec |\bv_f\circ\bx_s|_{H_{00}^{1/2}(\hat\Gamma)}\ec |\hat\bv_s|_{1/2,\partial\hat\Omega_s}.$$
Then 
$$\|\hat\bv_s\|^2_{1,\hat\Omega_s}\lesssim \|\hat\bv_s\|^2_{1/2,\partial\hat\Omega_s}\lesssim\|\bv_f\circ\bx_s\|^2_{H_{00}^{1/2}(\hat\Gamma)}\lesssim d^{N+2}_0d_1^2\|\bv_f\|^2_{1/2,\partial\Omega_f}\lesssim  d^{N+2}_0d^2_1\|\bv_f\|_{1,\Omega_f}.$$

Therefore, we have
$$\| \bv\|^2_1\lesssim d_0^{N+2}d_1^2\|\bv_f\|_{1,\Omega_f}^2$$
and

$$\frac{(\nabla\cdot \bv_f,q)_{\Omega_f}}{\| \bv\|_1\|q\|_0}\gtrsim  \frac{1}{d_0^{N/2+1}d_1}.$$

This finishes the proof.
\end{proof}

With the inf-sup condition of $b(\cdot,\cdot)$ proved, the well-posedness of (\ref{eq:saddle}) is shown.

\begin{theorem}
\label{thm:wellposed_continuous}
Assume that at a given time step $t^n$, %$$\bx_s\in W^{1,\infty}(\hat\Omega_s)\quad\mbox{and} \quad \inf_{\hat \bx\in \hat\Omega_s}\det (\nabla\bx_s) >0.$$Given
there exist positive constants $C_0$ and $C_1$ such that 
$$\sup_{\hat\bx\in \hat\Gamma}\|\nabla\bx_s(\hat\bx)\|_{2}\leq C_0, \quad \sup_{\hat\bx\in \hat\Gamma}\left\{\det(\nabla\bx_s(\hat\bx))^{-1}\right\}\leq C_1,$$
where the positive constants $C_0$ and $C_1$ are independent of material parameters and time step sizes. Then, under the norms $\|\cdot\|_V$ and $\|\cdot\|_Q$, the variational problem (\ref{eq:saddle}) is uniformly well-posed  with respect to material parameters and time step sizes.
\end{theorem}

\begin{proof}
We prove this theorem by verifying the Brezzi's conditions (\ref{eq:brezzi_a}) and (\ref{eq:brezzi_b}). 

The boundedness and coercivity of $a(\cdot,\cdot)$ are shown by (\ref{eq:a_coercive}) and the boundedness of $b(\cdot,\cdot)$ is shown by (\ref{eq:b_bounded}).  Therefore, we only need to prove the inf-sup condition of $b(\cdot,\cdot)$.

Due to the choice of the parameter $r$,  the following inequality holds
\begin{equation}
\label{eq:Vr}
\| \bv\|_{V}\lesssim r^{1/2}\| \bv\|_{1,\Omega},\quad\forall\bv\in\mathbb{V}.
\end{equation}

Based on Lemma \ref{lm:infsup_b_ale}, it indicates that
$$\inf_{q\in \mathbb{Q}}\sup_{\bv\in \mathbb{V}}\frac{(\nabla\cdot \bv,q)_{\Omega_f}}{\| \bv\|_V\|q\|_Q}\gtrsim  \frac{1}{d_0^{N/2+1}d_1}.$$
Since $d_0\leq \max\{C_0,1\}$, $d_1\leq\max\{C_1,1\}$ and $C_0$ and $C_1$ are independent of material parameters and time step sizes,  the inf-sup constant is uniformly bounded below.  Therefore, we have shown that (\ref{eq:saddle}) is uniformly well-posed with respect to material parameters $\rho_f$, $\hat\rho_s$, $\mu_f$, $\mu_s$ and $\lambda_s$ and time step size $k$.
\end{proof}

\subsubsection*{Applications in unsymmetric cases}

In the GCE scheme we are considering, convection terms are treated explicitly using (\ref{eq:explicit_convection}).  A more stable discretization is to linearize convection terms by Newton's method.  This adds unsymmetric terms to the variational problem

$$c(\bu,\bv)= \int_{\Omega_f}\rho_f(\bw\cdot\nabla)\bu_f\cdot \bv_f+\int_{\Omega_f}\rho_f(\bu_f\cdot\nabla)\bz\cdot \bv_f,$$
where $\bw$ and $\bz$ are functions obtained from previous iteration steps.

 With the new term $c(\bv,\phi)$ added, the following variational problem is also well-posed under certain assumptions
\begin{description}
\item[]
Find $\bv\in \mathbb{V}$ and  $p\in \mathbb{Q}$  such that
\begin{equation}
\label{eq:saddle_ac}
\left\{
\begin{aligned}
&a(\bv,\phi)+c(\bv,\phi)+b(\phi,p)&=&\langle \tilde f,\phi\rangle,&\forall& \phi\in \mathbb{V},\\
&b(\bv,q) &=&0,&\forall& q\in \mathbb{Q}.\\
\end{aligned}
\right.
\end{equation}
\end{description}

The well-posedness of (\ref{eq:saddle_ac}) requires the boundedness and coercivity of $a(\bu,\bv)+c(\bu,\bv)$.

First we have

\begin{equation*}
\begin{aligned}
\int_{\Omega_f}\rho_f(\bw\cdot\nabla)\bu_f\cdot \bv_f \leq C \left(\frac{k\rho_f}{\mu_f}\right)^{1/2}\|\bw\|_{\infty}\| \bu\|_{V}\|\bv\|_V\\ 
\end{aligned}
\end{equation*}
and
\begin{equation*}
\begin{aligned}
\int_{\Omega_f}\rho_f(\bu_f\cdot\nabla)\bz\cdot \bv_f \leq k \|\nabla\bz\|_{\infty}\|\bu\|_V\|\bv\|_V.\\
\end{aligned}
\end{equation*}

Then 
\begin{equation}
\label{eq:c_bounded}
c(\bu,\bv)\leq \left(  C\left(k\rho_f/\mu_f\right)^{1/2}\|\bw\|_\infty +  k\|\nabla\bz\|_\infty\right)\|\bu\|_V\|\bv\|_V.
\end{equation}

%On the other hand, $\forall u\in \Z$,
%\begin{equation}
%\begin{aligned}
%c(u,v)&\leq \left( \|\phi\|_{\infty}\|\nabla u\|_{\Omega_f}\|v\|_{\Omega_f}+\|\nabla \phi\|_\infty\|u\|_{\Omega_f} \|v\|_{\Omega_f}\right)\\
%&\leq \Delta t^{1/2}\|\phi\|_{\infty}(\Delta t^{-1}\|v\|_{\Omega_f}^2+\|\nabla u\|^2_{\Omega_f})+c_1\Delta t^2\|\nabla \phi\|_\infty\|u\|^2\\
%&\leq (\Delta t^{1/2}\|\phi\|_{\infty}+\Delta t^2\|\nabla \phi\|_\infty)\|u\|^2.\\
%\end{aligned}
%\end{equation}
%\begin{equation}
%\begin{aligned}
%c(u,u)&\leq \left( \|\phi\|_{\infty}\|\nabla u\|_{\Omega_f}\|u\|_{\Omega_f}+\|\nabla \phi\|_\infty\|u\|^2_{\Omega_f}\right)\\
%&\leq \Delta t^{1/2}\|\phi\|_{\infty}(\Delta t^{-1}\|u\|_{\Omega_f}^2+\|\nabla u\|^2_{\Omega_f})+d\Delta t\|\nabla \phi\|_\infty\|u\|^2\\
%&\leq (\Delta t^{1/2}\|\phi\|_{\infty}+d\Delta t\|\nabla \phi\|_\infty)\|u\|^2.\\
%\end{aligned}
%\end{equation}
%Here $c_1$ is a constant that does not depend on the solutions or parameters.

Assume $k$ is small enough such that

$$C\left(k\rho_f/\mu_f\right)^{1/2}\|\bw\|_\infty +  k\|\nabla\bz\|_\infty\leq c_0<1,$$ 
where $0<c_0<1$ is a constant.
%$\sqrt{k}\|\phi\|_{\infty}+k\|\nabla \phi\|_\infty<c_0=\min\{1,\mu_s,\mu_f\}$. 

Then we have the boundedness and coercivity  of $a(\bu,\bv)+c(\bu,\bv)$

\begin{equation}
\label{eq:ac_Velliptic}
\begin{aligned}
a(\bu,\bu)+c(\bu,\bu)\geq& (1-c_0)\|\bu\|_V^2,\quad\forall \bu\in \mathbb{V},\\
a(\bu,\bv)+c(\bu,\bv)\leq& (1+c_0)\|\bu\|_V\|\bv\|_V ,\quad\forall\bu,\bv\in \mathbb{V}.
\end{aligned}
\end{equation}

The boundedness and the inf-sup condition of $b(\cdot,\cdot)$ are not affected by $c(\cdot,\cdot)$. Therefore, the well-posedness of variational problem (\ref{eq:saddle_ac}) follows based on standard arguments. (See Corollary 4.1 in \cite{Girault.V;Raviart.P1986a}.)  We do not show the details here.  Although our study can be applied to unsymmetric case, we only deal with the symmetric cases in the rest of this paper.

  In the next section, we consider the well-posedness of the finite element problem (\ref{eq:fsi_mgce_fem}).

%
%By incorporating the conclusion from Theorem \ref{thm:infsup_b_ale}, we have the following corollary about the well-posedness of (\ref{eq:saddle_ge_mi})
%\begin{coro}
%\label{coro:fsi_convection}
%Assume $\cA=I$ and that $\|\phi\|_\infty$ and $\|\nabla\phi\|_{\infty}$ are finite. Then (\ref{eq:saddle_ge_mi}) is well-posed if 
% \begin{equation}
% \label{eq:assumption_xs}
% \frac{\|\det(\nabla{\bx_s})^{-1}\|_\infty}{(1+\|\nabla \bx_s\|_\infty)^2}\geq \beta>0
% \end{equation}
%and $k>0$ satisfies
%\begin{equation}
%\label{eq:assumption_phi}
% \left(\frac{k}{\mu_f\rho_f}\right)^{1/2}\|\phi\|_\infty +  \frac{dk}{\sqrt{\rho_f\rho_s}}\|\nabla\phi\|_\infty<1.
% \end{equation}
%
%%$$\sqrt{\Delta t}\|\phi\|_{\infty}+\Delta t^2\|\nabla \phi\|_\infty<\frac{\min\{1,\mu_s,\mu_f\}}{c_1}.$$
%
%Here $d$ is the dimension of the problem.
%\end{coro}
%\begin{proof}
%Based on Theorem \ref{thm:infsup_div},  boundedness of $a(\bu,\bv)$ and $c(\bu,\bv)$ in Equation (\ref{eq:a_bounded}) and (\ref{eq:c_bounded}) and the V ellipticity of $a(\bu,\bv)+c(\bu,\bv)$ in Equation (\ref{eq:ac_Velliptic}). Hence the proof follows based on standard argument. (See Corollary 4.1 in \cite{Girault.V;Raviart.P1986a}.)
%
%\end{proof}

\subsection{Well-posedness of finite element discretization}

Since we have already assumed $\mathbb{V}_h\subset\mathbb{V}$ and $\mathbb{Q}_h\subset\mathbb{Q}$, (\ref{eq:fsi_mgce_fem}) can be formulated as follows%Consider finite element spaces: $\mathbb{V}_h\subseteq \mathbb{V},$ $\mathbb{Q}_h\subseteq\mathbb{Q} .$  The following variational problem is a discretized version of (\ref{eq:saddle_ge_mi}):

\begin{description}
\item[] Find $\bv_h\in \mathbb{V}_h$ and  $p_h\in \mathbb{Q}_h$  such that
\begin{equation}
\label{eq:saddle_fem}
\left\{
\begin{aligned}
&a(\bv_h,\bphi_h)+b(\bphi_h,p_h)&=&\langle \tilde g,\bphi_h\rangle,&\forall& \bphi_h\in \mathbb{V}_h,\\
&b(\bv_h,q_h) &=&0,&\forall& q_h\in \mathbb{Q}_h.\\
\end{aligned}
\right.
\end{equation}
\end{description}

 The well-posedness of this finite element problem can be proved with some additional assumptions

The discrete kernel space is $$\mathbb{Z}_h:=\{\bv_h=(\bv_{h,f},\hat\bv_{h,s})\in \mathbb{V}_h| (\nabla\cdot \bv_{h,f},q_h)_{\Omega_f}=0, ~\mbox{ for all }q_h\in \mathbb{Q}_h\}.$$

As is pointed out in \cite{Xie.X;Xu.J;Xue.G2008a}, for finite element spaces that do not satisfy $\mathbb{Z}_h\subset\mathbb{Z}$, the uniform coercivity of $a(\cdot,\cdot)$ in $\mathbb{Z}_h$ cannot be guaranteed.  In fact, if $$r(\nabla\cdot\bv_f,\nabla\cdot\bv_f)_{\Omega_f}\leq a(\bv,\bv), \quad \mbox{ for all }\bv\in \mathbb{Z}_h$$ 
holds uniformly  with respect to $r$, then it implies that $\nabla\cdot\bv_f=0$ in $\Omega_f$, i.e. $\bv\in \mathbb{Z}$.  However, most commonly used finite element pairs do not satisfy $\mathbb{Z}_h\subset\mathbb{Z}$.  Although there are exceptions like P4-P3 in 2D, the choice is very restricted.   We propose two remedies for this issue: the first is to add a stabilization term to $a(\bu,\bv)$ and the second is to Introduce a new norm for $\mathbb{V}$.

\subsubsection{Remedy 1: Stabilized formulation for finite elements}

The first remedy we propose is to add the stabilization  term proposed in \cite{Xie.X;Xu.J;Xue.G2008a}  
$$\tilde a(\bu,\bv)=a(\bu,\bv)+r(\nabla\cdot\bu_f,\nabla\cdot\bv_f)_{\Omega_f}.$$
Then $\tilde a(\bu,\bv)$ is uniformly coercive in $\mathbb{V}_h$ since 
\begin{equation}
\label{eq:atilde_coercive}
\tilde a(\bu,\bu)\equiv\|\bu\|_V^2,\quad\forall\bu\in \mathbb{V}_h.
\end{equation}
 The stabilization term $r(\nabla\cdot\bu_f,\nabla\cdot\bv_f)_{\Omega_f}$ is one of the key ingredients in our formulation. This term has also been used in \cite{Olshanskii.M;Reusken.A2004a} to stabilize Stokes equations and the effects of this term on discretization error and preconditioning of the linear system are discussed.  Another type of stabilization technique, the \emph{orthogonal subgrid scales} technique, is applied to FSI in \cite{Badia.S;Quaini.a;Quarteroni.a2008b,Badia.S;Quaini.A;Quarteroni.A2008a} to stabilize the Navier-Stokes equations with equal-order velocity-pressure pairs (like P1-P1).  The stabilization parameters of this technique are determined by Fourier analysis in \cite{Codina.R2002a}. 

The new FEM problem is as follows:
\begin{description}
\item[] Find $\bv_h\in \mathbb{V}_h$ and  $p_h\in \mathbb{Q}_h$  such that
\begin{equation}
\label{eq:saddle_fem_stab}
\left\{
\begin{aligned}
&\tilde a(\bv_h,\bphi_h)+b(\bphi_h,p_h)&=&\langle \tilde g,\bphi_h\rangle,&\forall& \bphi_h\in \mathbb{V}_h,\\
&b(\bv_h,q_h) &=&0,&\forall& q_h\in \mathbb{Q}_h.\\
\end{aligned}
\right.
\end{equation}
\end{description}   

For this new formulation, we just need to prove the inf-sup conditions of $b(\cdot,\cdot)$ in order to show that it is well-posed.  Similar to Theorem \ref{thm:wellposed_continuous}, the inf-sup conditions of $b(\cdot,\cdot)$ also depend on $\bx_s$.  Note that $\bx_s$ is the solid trajectory and is calculated based on the solid velocity calculated at previous time steps. Moreover, $\bx_s$ corresponds to mesh motion and thus we assume that $\bx_s$ is piecewise linear on the triangulation.

\begin{coro}
\label{coro:infsup_b_fem}
Assume that $\bx_s$ is continuous and satisfies 

$$\bx_s|_{\tau}\in \mathcal{P}_1, ~~\forall \tau\in T_h(\hat\Omega_s)~~\mbox{ and } \inf_{\hat\bx\in \hat\Omega_s}\det (\nabla \bx_s) >0,$$
and that the finite element pair $(\mathbb{V}_{h,f},\mathbb{Q}_h)$ for the fluid variables satisfies that 
 \begin{equation}
 \label{eq:infsup_f}
 \inf_{q\in\mathbb{Q}_h}\sup_{\bv_f\in \mathbb{V}_{h,f}}\frac{(\nabla\cdot \bv_f,q)_{\Omega_f}}{\|\bv_f\|_1\|q\|_0}\gtrsim 1.
 \end{equation}
  Then the following inf-sup condition holds
\begin{equation}
\label{eq:infsup_b_fem}
\inf_{q\in\mathbb{Q}_h}\sup_{\bv\in \mathbb{V}_h}\frac{b( \bv,q)}{\|\bv\|_1\|q\|_0}\gtrsim \frac{1}{d_0^{N/2+1}d_1}.
\end{equation}
%$$\mathbf{div} \bv_f= q\quad, \text{ and } \|(\bv_f,\hat \bv_s)\|\lesssim  \left(\frac{1+\|\bx_s'\|^2}{c_0}\right)^{1/2}\|q\|.$$
\end{coro}
Note that $d_0$ and $d_1$ are defined in (\ref{eq:d0d1}).
\begin{proof}
Based on (\ref{eq:infsup_f}), we know that given any $q^h\in\mathbb{Q}_h$, we can find $\bv_f^h\in \mathbb{V}_{h,f}$ such that 
$$\frac{(\nabla\cdot \bv^h_f,q^h)_{\Omega_f}}{\|\bv_f^h\|_1}\gtrsim \|q^h\|_0.$$

We take $\hat \bv_s^h$ such that $\hat\bv_s^h=\bv_f^h\circ\bx_s^h$ on $\hat\Gamma$ and
$$\int_{\hat\Omega_s}\nabla\hat\bv_s^h:\nabla\bphi_h=0,\quad\forall \bphi_h\in \mathbb{V}_{h,s}^0,$$
where $\mathbb{V}_{h,s}^0:=\{\bv\in \mathbb{V}_{h,s}| \bv=0, \mbox{ on }\partial\hat\Omega\}.$  This discrete harmonic extension $\hat\bv_{s}^h$ still satisfies
$$\|\hat\bv_s^h\|_{1,\hat\Omega_s}\lesssim \|\hat\bv_s^h\|_{1/2,\partial\hat\Omega_s}$$ since $\hat\bv_s^h$ is the projection of the continuous harmonic extension (see (\ref{eq:harmonic_extension})) under the inner product $(\nabla\bu,\nabla\bv).$

%Similar with the proof of Lemma \ref{lm:infsup_b_ale} , we have
%
%$$\|\hat\bv_s\|^2_{1,\hat\Omega_s}\lesssim  d^{n+2}_0d_1^2\|\bv_f^h\|_{1,\Omega_f}.$$
%We take $\hat\bv_s^h$ to be the Scott-Zhang interpolation of $\hat\bv_s$ and thus
%$$\|\hat\bv_s^h\|_1\lesssim \|\hat\bv_s\|_1.$$
%

Then, take $\bv^h=(\bv_f^h,\hat\bv_s^h)\in \mathbb{V}_h$. We know that
$$\| \bv^h\|^2_1\lesssim d_0^{N+2}d_1^2\|\bv_f^h\|_{1}^2$$
and, therefore, the following inequality holds

$$\frac{(\nabla\cdot \bv^h,q^h)_{\Omega_f}}{\|\bv^h\|_1}\gtrsim  \frac{\|q^h\|_0}{d_0^{N/2+1}d_1} .$$
This finishes the proof.

\end{proof}

With the inf-sup condition of $b(\cdot,\cdot)$ proved, the well-posedness  of (\ref{eq:saddle_fem}) follows.
\begin{theorem}
\label{thm:wellposed_discrete}
Assume that the assumptions in Corollary \ref{coro:infsup_b_fem} hold and that at a given time step $t^n$, there exist constants $C_0$ and $C_1$ such that
$$\sup_{\hat\bx\in \hat\Gamma}\|\nabla\bx_s(\hat\bx)\|_{2}\leq C_0, \quad \sup_{\hat\bx\in \hat\Gamma}\left\{\det(\nabla\bx_s(\hat\bx))^{-1}\right\}\leq C_1.$$
Moreover, assume that $C_0$ and $C_1$ are independent of material and discretization parameters. Then, under the norms $\|\cdot\|_V$ and $\|\cdot\|_Q$ the stabilized variational problem (\ref{eq:saddle_fem_stab}) is uniformly well-posed with respect to material and discretization parameters.
\end{theorem}

\begin{proof}
To prove this theorem we also verify the Brezzi's conditions. 

The boundedness and coercivity of $\tilde a(\cdot,\cdot)$ is obvious due to (\ref{eq:atilde_coercive}). The boundedness of $b(\cdot,\cdot)$ can be similarly proved by $(\ref{eq:b_bounded})$.  Corollary \ref{coro:infsup_b_fem} proves 
$$
\inf_{q\in\mathbb{Q}_h}\sup_{\bv\in \mathbb{V}_h}\frac{b(\bv,q)}{\|\bv\|_1\|q\|_0}\gtrsim \frac{1}{d_0^{N/2+1}d_1}.
$$
Since (\ref{eq:Vr}) still holds for $\bv\in\mathbb{V}_h$, the following inf-sup condition is proved
$$
\inf_{q\in\mathbb{Q}_h}\sup_{\bv\in \mathbb{V}_h}\frac{b(\bv,q)}{\|\bv\|_V\|q\|_Q}\gtrsim \frac{1}{d_0^{N/2+1}d_1}.
$$
Moreover, the inf-sup constant $d_0^{-N/2-1}d_1^{-1}$ is uniformly bounded below due to $d_0\leq\max\{C_0,1\}$ and $d_1\leq\max\{C_1,1\}$ .
We have verified all the Brezzi's conditions and all of the inequalities hold uniformly with respect to material parameters $\rho_f$, $\hat\rho_s$, $\mu_f$, $\mu_s$ and $\lambda_s$, time step size $k$ and mesh size. Therefore, (\ref{eq:saddle_fem}) is uniformly well-posed with respect to material and discretization parameters.
\end{proof}

\subsubsection{Remedy 2: A new norm for $\mathbb{V}$}

An equivalent form of the norm $\|\cdot\|_V$ is
%\begin{equation}
%\begin{aligned}
%\forall\bu\in \mathbb{V},~~\|\bu\|_{V_Q}^2&:=a(\bu,\bu)+r\left(\sup_{q\in \mathbb{Q}}\frac{(\nabla\cdot \bu_f,q)_{\Omega_f}}{\|q\|_{0,\Omega_f}}\right)^2.\\
%%r\|\nabla\cdot\bu\|_{0,\Omega_f}^2,\\
%\end{aligned}
%\end{equation}
%
%This norm can also be written as
$$\mbox{for all }\bu\in \mathbb{V},~~\|\bu\|_{V_Q}^2:=a(\bu,\bu)+r\|\mathcal{P}_{\mathbb{Q}}\nabla\cdot\bu_f\|_{0,\Omega_f}^2,$$
where $\mathcal{P}_{\mathbb{Q}}$ is the $L^2$ projection from $L^2(\Omega_f)$ to $\mathbb{Q}$. This norm was used in \cite{Benzi.M;Olshanskii.M2011a} to study the well-posedness of linearized Navier-Stokes equations.

Note that this norm depends on the choice of space $\mathbb{Q}$ and we use the subscript $V_Q$ to emphasize that.  For $\mathbb{Q}=L^2(\Omega_f)$, we have $\|\bu\|_{V}=\|\bu\|_{V_Q}$, for all $\bu\in \mathbb{V}$. For finite element pair $(\mathbb{V}_h,\mathbb{Q}_h)$, the norm is 
$$\forall\bu\in \mathbb{V}_h,~~\|\bu\|_{V_{Q}}^2=a(\bu,\bu)+r\|\mathcal{P}_{\mathbb{Q}_h}\nabla\cdot\bu_f\|_{0,\Omega_f}^2.\\$$

 With this new norm, we prove the well-posedness of the original finite element discretization (\ref{eq:saddle_fem}) without adding the stabilization term $r(\nabla\cdot\bu_f,\nabla\cdot\bv_f)_{\Omega_f}$.
 \begin{theorem}
 \label{thm:wellposed_discrete_newnorm}
 Assume that the assumptions in Theorem \ref{thm:wellposed_discrete} hold.  Then, under the norms $\|\cdot\|_{V_{Q}}$ and $\|\cdot\|_Q$ the original variational problem (\ref{eq:saddle_fem}) is uniformly well-posed with respect to material and discretization parameters.
\end{theorem}

\begin{proof}
Note that under the new norm $\|\cdot\|_{V_Q}$, $a(\cdot,\cdot)$ is uniformly coercive in $\mathbb{Z}_h$. In fact,
$$\mbox{for all }\bu\in \mathbb{Z}_h, ~~ a(\bu,\bu)=\|\bu\|_{V_Q}^2.$$
The boundedness of $a(\cdot,\cdot)$ is obvious.  The boundedness of $b(\cdot,\cdot)$ is also easy to show:
$$b(\bv_f,p)=(\nabla\cdot\bv_f,p)_{\Omega_f}\leq \|p\|_{0,\Omega_f}\sup_{q\in\mathbb{Q}_h}\frac{(\nabla\cdot\bv_f,q)_{\Omega_f}}{\|q\|_{0,\Omega_f}}\leq \|p\|_Q\|\bv\|_{V_{Q_h}}.$$
Since
$$\|\bv\|_{V_Q}\lesssim r^{1/2}\|\bv\|_{1,\Omega}$$
is still valid, the inf-sup conditions of $b(\cdot,\cdot)$ can be proved by using Corollary \ref{coro:infsup_b_fem}.  This concludes our proof.
\end{proof}

We have provided two remedies in order to get uniformly well-posed finite element discretizations. In the next section, we introduce how these stable formulations can help us find optimal preconditioners.

\section{Solution of linear systems}

In this section, we consider preconditioners for (\ref{eq:saddle_fem}).  Define $\mathbb{X}_h=\mathbb{V}_h\times\mathbb{Q}_h$.  The underlying norm is
$$\|(\bv,p)\|^2_{X}=\|\bv\|_V^2+\|p\|_Q^2,\quad (\bv,p)\in \mathbb{X}_h.$$ 

%For $x=(\bv,p),y=(\bphi,q)$, $A(\cdot,\cdot)$ is defined by
%$$A(x,y)=\tilde a(\bv,\bphi)+b(\bphi,p)+b(\bv,q).$$

 Consider the following saddle point problem:

\begin{description}
%\begin{prob}
%\label{pr:var_c}
\item[]
Find $x\in \mathbb{X}_h$, such that
\begin{equation}
\label{eq:saddle_A}
K(x,y)=\langle\tilde g,y\rangle,\quad \forall y\in \mathbb{X}_h,
\end{equation}
%\end{prob}
\end{description}
where $\tilde g\in \mathbb{X}_h'$.
The operator form of (\ref{eq:saddle_A})  is

$$\cK_h x=\tilde g.$$

Under the assumption that (\ref{eq:saddle_A}) is uniformly well-posed, an optimal preconditioner can be found \cite{Mardal.K;Winther.R2011a,Zulehner.W2011a}, which is the Riesz operator $\cB_h:\mathbb{X}_h'\mapsto \mathbb{X}_h$ defined by

$$(\cB_h f,y)_X=\langle f,y\rangle,\quad \forall y\in \mathbb{X}_h, f\in \mathbb{X}_h'.$$
Thus, $\cB_h$ satisfies  
$$\kappa(\cB_h\cK_h)\lesssim 1.$$

The uniform boundedness of the condition number $\kappa(\cB_h\cK_h)$ results in uniform convergence of Krylov subspace methods, such as MINRES.   
\subsection{Two optimal preconditioners for FSI}
In the previous section, we have introduced two stable finite element formulations, which provide two optimal preconditioners.  To facilitate our discussion, we first introduce the block matrices $A_h$, $D_h$, $B_h$, defined by
\begin{equation*}
\begin{aligned}
( A_h\bar{u}_h,\bar{v}_h)&=a(\bu_h,\bv_h),\\
(B_h\bar{u}_h,\bar{p}_h)&=b(\bu_h,q_h),\\
(D_h\bar{u}_h,\bar{v}_h)&=(\nabla\cdot \bu_{h,f},\nabla\cdot\bv_{h,f})_{\Omega_f},\\
\end{aligned}
\end{equation*}
for any $\bu_h$, $\bv_h\in \mathbb{V}_h$ and $p_h\in \mathbb{Q}_h$. $\bar u_h, \bar v_h$ and $\bar p_h$ are the corresponding vector representations with given bases for $\mathbb{V}_h$ and $\mathbb{Q}_h$.  We also introduce the pressure mass matrix $M_p$.

Now, we introduce two optimal preconditioning strategies ({\bf M1}) and ({\bf M2}) based on the uniformly well-posed formulations introduced in the previous section. Note that these two preconditioners are applied to (\ref{eq:saddle_fem}) and (\ref{eq:saddle_fem_stab}), respectively.
\begin{itemize}
\item Formulation 1 ({\bf M1}): With the stabilization term added, (\ref{eq:saddle_fem_stab}) is uniformly well-posed under the norms $\|\cdot\|_V$ and $\|\cdot\|_Q$.  In this case, 
$$K(x,y)=\tilde a(\bv,\bphi)+b(\bphi,p)+b(\bv,q),$$
where $x=(\bv,p)$ and $y=(\bphi,q)$.

The optimal preconditioner in this case is
\begin{equation}
\label{eq:precond_B1}\cB_h^1=
\left(
\begin{array}{cc}
 A_h+rD_h &0\\
0& \frac{1}{r}M_p\\
\end{array}
\right)^{-1}.
\end{equation}

\item Formulation 2 ({\bf M2}): With the new norm $\|\cdot\|_{V_Q}$ introduced,  (\ref{eq:saddle_fem}) is uniformly well-posed under the norms $\|\cdot\|_{V_Q}$ and $\|\cdot\|_{Q}.$ In this case, 
$$K(x,y)=a(\bv,\bphi)+b(\bphi,p)+b(\bv,q),$$
where $x=(\bv,p)$ and $y=(\bphi,q)$.

Given $p_h\in \mathbb{Q}_h$ and $\bv_h\in \mathbb{V}_h$ satisfying $p_h=\mathcal{P}_Q(\nabla\cdot\bv_h)$, we know that 
$$M_p\bar p_h=B_h\bar v_h.$$ 
Therefore,
$$\|p_h\|_{0,\Omega_f}^2=\bar p_h^TM_p\bar p_h=\bar v_h^T B_h^TM_p^{-1}B_h\bar v_h.$$
Then we know that the corresponding optimal preconditioner  in this case is
\begin{equation}
\label{eq:precond_B2}\cB_h^2=
\left(
\begin{array}{cc}
A_h+rD_h^Q &0\\
0& \frac{1}{r}M_p\\
\end{array}
\right)^{-1},
\end{equation}
where $D_h^Q:=B_h^TM_p^{-1}B_h$.
\end{itemize}

\subsection{Comparing {$\cB_h^1$}, {$\bf \cB_h^2$} and the augmented Lagrangian (AL) preconditioner}

The AL preconditioner was proposed for Oseen problems in \cite{Benzi.M;Olshanskii.M2006a} and has been extended to the Navier-Stokes equations in \cite{Benzi.M;Olshanskii.M;Wang.Z2011a,Benzi.M;Olshanskii.M2011a}.   The AL preconditioner is designed for saddle point problems of the following form
\begin{equation}
\label{eq:saddle_block}
\left(
\begin{array}{cc}
A & B^T\\
B&0\\
\end{array}
\right)
\left(
\begin{array}{c}
u\\
p\\
\end{array}
\right)
=
\left(
\begin{array}{c}
f\\
0\\
\end{array}
\right).
\end{equation}
The AL preconditioner is applied to the modified saddle point problem
\begin{equation}
\label{eq:saddle_block_AL}
\left(
\begin{array}{cc}
A +\gamma B^TW^{-1}B& B^T\\
B&0\\
\end{array}
\right)
\left(
\begin{array}{c}
u\\
p\\
\end{array}
\right)
=
\left(
\begin{array}{c}
f\\
0\\
\end{array}
\right),
\end{equation}
and the ideal  form of the AL preconditioner is
\begin{equation}
P_\gamma=\left(
\begin{array}{cc}
A_\gamma & B^T\\
0&\frac{1}{\nu+\gamma}W\\
\end{array}
\right)^{-1},
\end{equation}
where $A_\gamma=A +\gamma B^TW^{-1}B$, $\nu$ is the kinematic viscosity, and the ideal choice of $W$ is the pressure mass matrix $M_p$.   
Note that (\ref{eq:saddle_block}) and (\ref{eq:saddle_block_AL}) have the same solution.

Practical choices for the preconditioner $P_{\gamma}$ are discussed extensively in literature, though we do not discuss this issue here.  For the application to the Oseen problem\cite{Benzi.M;Olshanskii.M2006a}, eigenvalue analysis shows that the preconditioned matrix has all the eigenvalues tend to $1$ as $\gamma$ tends to $\infty$.  In the application to linearized Navier-Stokes problem \cite{Benzi.M;Olshanskii.M2011a}, it is shown that for certain choices of the parameter $\gamma$, the convergence rate of AL-preconditioned GMRes is independent of discretization and material parameters.   Note that in these applications, convection terms are considered and, therefore, the linear systems are not symmetric.

The AL preconditioning technique can also be applied to our FSI problem.  By simply adding the term 
$r(P_Q\nabla\cdot\bu_f,\nabla\cdot\bv_f)_{\Omega_f}$ (or $rB^TW^{-1}B$ in matrix form) to the first equation of $(\ref{eq:saddle_fem})$, the resultant variational problem
\begin{description}
\item[] Find $\bv_h\in \mathbb{V}_h$ and  $p_h\in \mathbb{Q}_h$  such that
\begin{equation}
\label{eq:saddle_fem_AL}
\left\{
\begin{aligned}
&a(\bv_h,\bphi_h)+r(P_Q\nabla\cdot\bu_f,\nabla\cdot\bv_f)_{\Omega_f}+b(\bphi_h,p_h)&=&\langle \tilde g,\bphi_h\rangle,&\forall& \bphi_h\in \mathbb{V}_h,\\
&b(\bv_h,q_h) &=&0,&\forall& q_h\in \mathbb{Q}_h,\\
\end{aligned}
\right.
\end{equation}
\end{description}
is also uniformly well-posed under the norms $\|\cdot\|_{V_Q}$ and $\|\cdot\|_Q$ since adding this term  yields
$$a(\bu,\bu)+r(P_Q\nabla\cdot\bu_f,\nabla\cdot\bu_f)_{\Omega_f}=\|\bu\|_{V_Q}^2,~~\forall\bu\in\mathbb{V}_h,$$
and the boundedness and the inf-sup condition of $b(\cdot,\cdot)$ still hold.  Based on this observation, we propose the third optimal preconditioning strategy ({\bf M3}), which is very similar to the AL preconditioner.
\begin{itemize}
\item Formulation 3 ({\bf M3}): 
We take the following bilinear form $K(\cdot,\cdot)$ for the saddle point problem (\ref{eq:saddle_A})
$$K(x,y)=a(\bv,\bphi)+r(P_Q\nabla\cdot\bv_f,\nabla\cdot\bphi_f)_{\Omega_f}+b(\bphi,p)+b(\bv,q),$$
where $x=(\bv,p)$ and $y=(\bphi,q)$.

The optimal preconditioner in this case is also $\mathcal{B}_h^2$.
\end{itemize}

By using $\mathcal{B}_h^2$ in an upper triangular fashion, it becomes quite similar to the AL preconditioner.  Therefore, our analysis can also provide justification for the AL-type preconditioner for FSI in the absence of the convection term.  Note that the choice of parameters (in terms of $r$) in (\ref{eq:precond_B2}) is different from those used in AL precondtioners in the literature.

We compare the preconditioning techniques ({\bf M1}), ({\bf M2}) and ({\bf M3}) in the Table \ref{tb:comparison}.  All of these three preconditioners are similar to the velocity Schur complement preconditioners.  For comparison, we also list a pressure Schur complement (SC) preconditioner in Table \ref{tb:comparison}.

\begin{table}[htdp]
\caption{Compare {\bf M1}, {\bf M2}, {\bf M3} and SC} 
\begin{center}
\begin{tabular}{|c|c|c|}
\hline
 		&preconditioner	& stiffness matrix\\
\hline
{\bf M1}&$\left(
\begin{array}{cc}
A_h+rD_h &0\\
0& \frac{1}{r}M_p\\
\end{array}
\right)^{-1}$ 
&
$
\left(
\begin{array}{cc}
A_h+rD_h &B_h^T\\
B_h& 0\\
\end{array}
\right)
$
\\
\hline
{\bf M2}& $\left(
\begin{array}{cc}
A_h+rD_h^Q &0\\
0& \frac{1}{r} M_p\\
\end{array}
\right)^{-1}$
&
$
\left(
\begin{array}{cc}
A_h &B_h^T\\
B_h& 0\\
\end{array}
\right)
$
 \\	
\hline
{\bf M3} &
$\left(
\begin{array}{cc}
A_h+rD_h^Q &0\\
0& \frac{1}{r} M_p\\
\end{array}
\right)^{-1}$
&
$\left(
\begin{array}{cc}
A_h+rD_h^Q &B_h^T\\
B_h&0\\
\end{array}
\right)$
\\
\hline
SC&
$
\left(
\begin{array}{cc}
A_h &B_h^T\\
0& -B_hA_h^{-1}B_h^{T}\\
\end{array}
\right)^{-1}
$
&
$
\left(
\begin{array}{cc}
A_h &B_h^T\\
B_h& 0\\
\end{array}
\right)
$
 \\
\hline
\end{tabular}
\end{center}
\label{tb:comparison}
\end{table}%

Note that in the pressure Schur complement preconditioner (SC),  we use the inverse of the diagonal part of $A_h$ to approximate $A_h^{-1}$.
\begin{remark}

\begin{itemize}
\item Adding the term $r(\nabla\cdot\bu_f,\nabla\cdot\bv_f)_{\Omega_f}$ to the continuous problem (\ref{eq:saddle}) does not change the solution. But adding it may change the solution of finite element discretized problems; thus, (\ref{eq:saddle_fem}) and (\ref{eq:saddle_fem_stab}) may have different solutions, especially when $r$ is large.  In comparison, {\bf M2} and {\bf M3} do not change the solutions of finite element problems.
\item {\bf M2} and {\bf M3} have very similar forms. They differ in that {\bf M2} does not add $rD_h^Q$ to the stiffness matrix.

\item {\bf M1}, {\bf M2} and {\bf M3} are all proven to be optimal for FSI based on our analysis.

\end{itemize}
\end{remark}

For the practical implementation, the performance of these preconditioners also depends on the efficiency of inverting the diagonal blocks, such as $A_h+rD_h$ and $M_p$.  The mass matrix $M_p$ is easy to invert by iterative methods. The velocity block $A_h$ is symmetric positive definite for the FSI problem;  Krylov subspace method preconditioned by multigrid is usually one of the most efficient solvers. However, there are still some difficulties that need special consideration:
\begin{itemize}

\item  The different scales of the fluid and structure problems result in large jumps in coefficients.  For example, the material parameters $\mu_s$ and $\mu_f$ can differ greatly in magnitude.  This leads to the following general jump-coefficient problem:

$$\mbox{Find } \bu\in H^1_0(\Omega) \mbox{ such that } ~~a(\bu,\bv)=\langle f,\bv\rangle, ~~\mbox{ for all } \bv\in H^1_0(\Omega),$$
where $a(\bu,\bv)=(\alpha(\bx)\epsilon(\bu),\epsilon(\bv))+(\beta(\bx)\nabla\cdot\bu,\nabla\cdot\bv)+(\gamma(\bx)\bu,\bv)$. The domain $\bar\Omega=\bar\Omega_1\cup\bar\Omega_2$  is illustrated in Figure \ref{fig:two_domain}.

\begin{figure}[h]
\setlength{\unitlength}{0.36in} % selecting unit length 
\centering 
% used for centering Figure 
\begin{picture}(10,4) % picture environment with the size (dimensions)
% 32 length units wide, and 15 units high. 
\put(3,0){\framebox(4,2){$\Omega_1$}}
\put(3,2){\framebox(4,2){$\Omega_2$}} 
\end{picture}
\caption{The domain for the jump-coefficient problem} % title of the Figure
 \label{fig:two_domain}
 \end{figure}
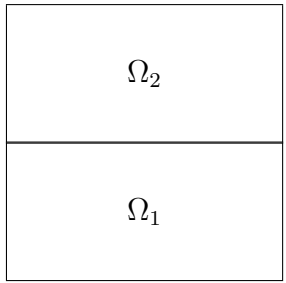

The coefficients $\alpha(\bx),\beta(\bx)$ and $\gamma(\bx)$ are piecewise positive constants on $\Omega_i$ $(i=1,2)$.  The question is how to design solvers that are robust with respect to the jumps of $\alpha(\bx)$, $\beta(\bx)$ and $\gamma(\bx)$.
 There is much research work on solving jump-coefficient problems. We refer to \cite{Xu.J;Zhu.Y2008a} and the references therein for related discussions.

\end{itemize}

\subsection{Numerical Examples}

In this section, we present some numerical experiments in order to verify our analysis.  Preconditioning techniques {\bf M1}, {\bf M2}, {\bf M3} and the SC preconditioner are tested.  

\begin{figure}
\begin{center}
\includegraphics[width=0.6\textwidth]{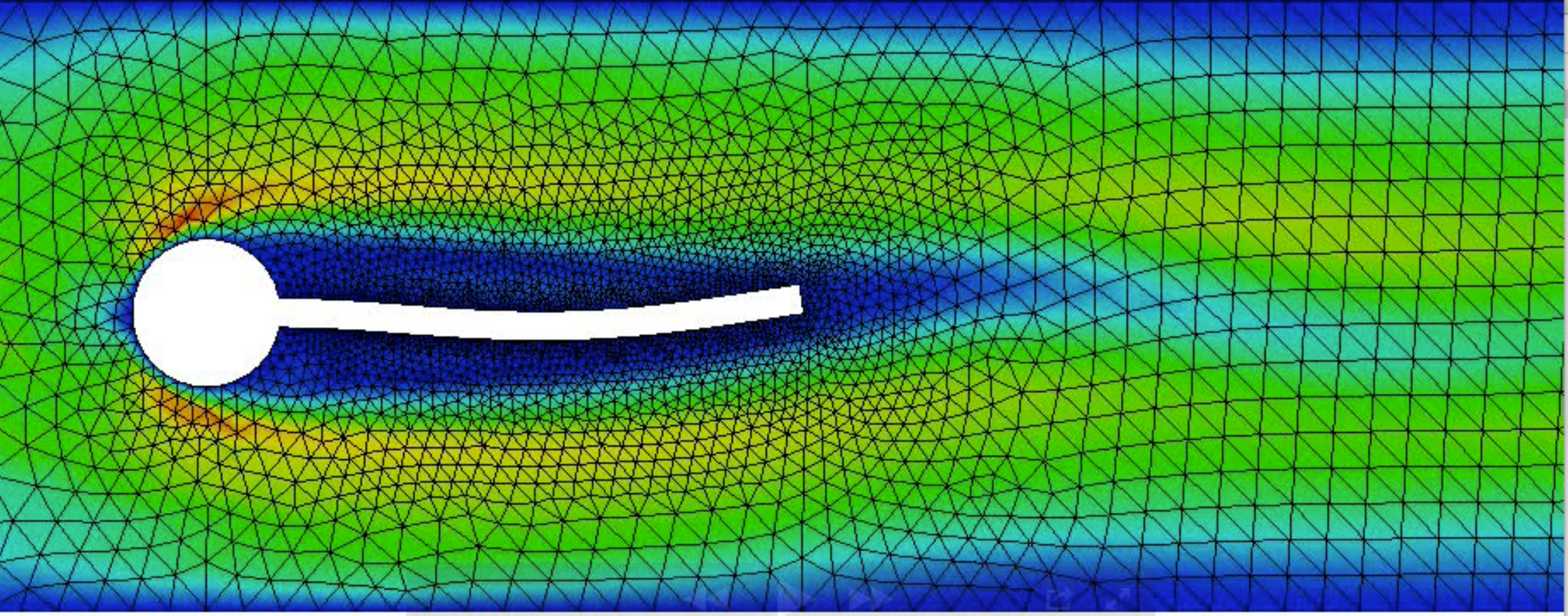}  
\end{center}
\caption{FSI benchmark problem} % title of the Figure
 \label{fig:fsi} % label to refer figure in text 
\end{figure}
We use the data from the FSI benchmark problem in \cite{Turek.S;Hron.J2006a}.  Note that this is a 2D problem. The FSI code is implemented in the framework of FEniCS\cite{Logg.A;Mardal.K;Wells.G;others2012a}. The computational domain is shown in Figure \ref{fig:fsi}.  We have an elastic beam in a channel, where the inflow comes from the left end of the domain. We prescribe zero Dirichlet boundary conditions on the top and bottom of the channel. On the right end we use no-flux boundary condition.   We use P2-P0 finite elements for the FSI system.

%

%
%\begin{figure}[h]
%\label{fig:fsi}
%\setlength{\unitlength}{0.14in} % selecting unit length 
%\centering 
%% used for centering Figure 
%\begin{picture}(16,8) % picture environment with the size (dimensions)
%% 32 length units wide, and 15 units high. 
%\put(-5,0){\framebox(30,6.1)}
%\put(0.5,2.7){\framebox(5,0.5)}
%\put(0,3){\circle{1}}
%\end{picture}
%\caption{Computational domain of the FSI benchmark problem} % title of the Figure
% \label{fig} % label to refer figure in text 
% \end{figure}
  We use three meshes with different sizes.  Numbers of degrees of freedom for these meshes are shown in Table \ref{tb:meshsize}.

\begin{table}[htdp]
\caption{DoFs of the meshes}
\begin{center}
\begin{tabular}{|c|c|c|c|}
\hline 
		&  mesh 1  &  mesh 2  & mesh 3 \\ \hline
	%size	&                 &                  &                \\ \hline
	DoF  &     11,714  &  45,932   &    181,880    \\ \hline 
\end{tabular}
\end{center}
\label{tb:meshsize}
\end{table}%

%To show the effect of the stabilization term $r(\nabla\cdot \bu,\nabla\cdot\bv)$, we solve the linear systems resulting from the original FSI formulation and the new formulation we proposed with stabilization term.  Stabilized $P1-P1$ finite element is used for fluid and $P1$ is used for structure displacement. 

 The values of the parameter $r$ in {\bf M1}, {\bf M2} and {\bf M3} are the same and are calculated by (\ref{eq:r}). Preconditioned GMRes is used to solve the linear systems. Although {\bf M1}, {\bf M2} and {\bf M3} are originally block diagonal preconditioners, we use them in a block upper triangular fashion.  Each of the diagonal blocks is solved exactly.  The iteration of GMRes stops when the relative residual has magnitude less than $10^{-10}$.  

In Table \ref{tb:deltat}, we test the preconditioners for different meshes and time step sizes.  In Table \ref{tb:density}, we show the test results for different meshes and density ratios.

\begin{table}[htdp]
\caption{Number of iterations for preconditioned GMRES for different time step sizes ($k=0.01, 0.001, 0.0001$)}
\begin{center}
\begin{tabular}{|c|c|c|c|c|c|c|c|c|c|c|c|c|}
\hline 
	       &   \multicolumn{4}{c|}{$k=0.01$} & \multicolumn{4}{c|}{$k=0.001$} & \multicolumn{4}{c|}{$k=0.0001$} \\ \hline
  preconditioner              &  {\bf M1}   &  {\bf M2}  &  {\bf M3} & SC& {\bf M1}   &  {\bf M2}  &  {\bf M3} &  SC&{\bf M1}   &  {\bf M2}  &  {\bf M3}&  SC   \\ 
                \hline	
  mesh 1   &   2       & 20      &    6    &  16  &     1      &19   &8     & 10 & 1	&26 & 7    & 9\\  \hline 
  mesh 2   &    2      &   20       &   6    &  26&    1     &  19   &  8  & 12   &1    &17& 8   &  9\\ \hline   
  mesh 3   &     2     &  24        &7       &  54&    1      &   23&  9    &    21   &1	&27&8     & 11\\ \hline    
\end{tabular}
\end{center}
\label{tb:deltat}
\end{table}%

\begin{table}[htdp]
\caption{Number of iterations for preconditioned GMRES for varying density ratios}
\begin{center}
\begin{tabular}{|c|c|c|c|c|c|c|c|c|c|c|c|c|}
\hline
	       &   \multicolumn{4}{c|}{$\hat\rho_s=\rho_f$} & \multicolumn{4}{c|}{$\hat\rho_s=10\rho_f$} & \multicolumn{4}{c|}{$\hat\rho_s=100\rho_f$} \\ \hline
     preconditioner           &  {\bf M1}   &  {\bf M2}  &  {\bf M3} & SC& {\bf M1}   &  {\bf M2}  &  {\bf M3} &  SC&{\bf M1}   &  {\bf M2}  &  {\bf M3}  &SC \\ 
                \hline	
  mesh 1   &  5        &13       &6      &18     &2    & 20    &6  & 16   &    2    & 25  &6  &15   \\  \hline 
  mesh 2   &     5    &   21       &  6    &31   & 2    & 20   &6  &  26  &   2 &    25 &5  &26   \\ \hline   
  mesh 3   &      5   &  25        &   7    &  61   &2  &24 	&7 &  54  &   2        & 26 & 5& 53 \\ \hline    
\end{tabular}
\end{center}
\label{tb:density}
\end{table}%

 From the data we see that the convergence of preconditioned GMRes for {\bf M1}, {\bf M2} and {\bf M3} is almost uniform  and quite robust for different mesh sizes, time step sizes, and density ratios.  The case with SC shows dependence on mesh sizes and the dependence becomes more significant when the time step size $k$ grows.     {\bf M1} and {\bf M3} in general need significantly fewer number of iterations than {\bf M2} and are more stable than {\bf M2} for various combinations of material and discretization parameters.

\section*{Concluding remarks}

In this paper, we formulate the FSI discretized system as saddle point problems. Under mild assumptions, the uniform well-posedness of the saddle point problems is shown. By adding a stabilization term or adopting a new norm for velocity, the finite element discretization of the FSI problem is also proved to be uniformly well-posed.  Two optimal preconditioners are proposed based on the well-posed formulations.  Our theoretical framework also provides an alternative justification for the AL-type preconditioners in the absence of the convection term.   In the numerical examples, we show the robustness of these preconditioners.  We use direct solves for the sub-blocks. In practice, these sub-blocks have to be inverted by iterative methods when their sizes are large. Robust preconditioners for the sub-blocks have to be considered.  
\section*{Acknowledgements}
We appreciate the contributions to the numerical tests from Dr. Xiaozhe Hu, Dr. Pengtao Sun, Feiteng Huang, and Lu Wang and many suggestions from  Dr. Shuo Zhang, Dr. Xiaozhe Hu, and Dr. Maximilian Metti, which have greatly improved the presentation of this paper. We also appreciate the helpful suggestions from Professor Alfio Quarteroni and Dr. Simone Depairs during the visit of the second author to EPFL.

}

\bibliographystyle{plain}
\bibliography{main}

\end{document}